\documentclass[11pt]{amsart}

\usepackage[
paper=a4paper,
headsep=17pt,footskip=17pt,margin=26mm
]{geometry}

\usepackage{float,array,multirow,calc,booktabs}

\usepackage[bookmarks]{hyperref}
\hypersetup{colorlinks=true,linkcolor=blue,citecolor=blue}

\usepackage{amssymb,amsxtra,bbm}
\usepackage[shortlabels]{enumitem}
\setlist[enumerate,1]{label=\textup{(\arabic*)}}

\usepackage{graphicx,color}

\usepackage{tikz}
\usetikzlibrary{
  cd,
  calc,
  positioning,
  arrows,
  decorations.pathreplacing,
  decorations.markings,
  intersections,
  patterns,
  bbox,
} 

\usepackage[mode=buildnew]{standalone}

\definecolor{todo-background-color}{gray}{0.95} 
\usepackage[
  textwidth=1.1in,
  backgroundcolor=todo-background-color,
  bordercolor=black,
  linecolor=black,
  textsize=footnotesize
]{todonotes}

\usepackage[final,color,notref,notcite]{showkeys}
\usepackage[color,notref,notcite]{showkeys}
\usepackage{seqsplit}

\definecolor{labelkey}{rgb}{0,.5,0}

\iftrue
\makeatletter
\def\@settitle{%
  \begin{flushleft}%
    \LARGE\bfseries
    \strut\@title\strut
  \end{flushleft}%
}
\def\@setauthors{%
  \begingroup
  \def\thanks{\protect\thanks@warning}%
  \trivlist
  \raggedright
  \large \@topsep28\p@\relax
  \advance\@topsep by -\baselineskip
  \item\relax
  \author@andify\authors
  \def\\{\protect\linebreak}%
  \authors
  \ifx\@empty\contribs
  \else
    ,\penalty-3 \space \@setcontribs
    \@closetoccontribs
  \fi
  \normalfont
  \endtrivlist
  \endgroup
}
\def\@setaddresses{\par
  \nobreak \begingroup
  \small\raggedright
  \def\author##1{\nobreak\addvspace\smallskipamount}%
  \def\\{\unskip, \ignorespaces}%
  \interlinepenalty\@M
  \def\address##1##2{\begingroup
    \par\addvspace\bigskipamount\noindent
    \@ifnotempty{##1}{(\ignorespaces##1\unskip) }%
    {\ignorespaces##2}\par\endgroup}%
  \def\curraddr##1##2{\begingroup
    \@ifnotempty{##2}{\nobreak\noindent\curraddrname
      \@ifnotempty{##1}{, \ignorespaces##1\unskip}\/:\space
      ##2\par}\endgroup}%
  \def\email##1##2{\begingroup
    \@ifnotempty{##2}{\nobreak\noindent E-mail address%
      \@ifnotempty{##1}{, \ignorespaces##1\unskip}\/:\space
      \ttfamily##2\par}\endgroup}%
  \def\urladdr##1##2{\begingroup
    \def~{\char`\~}%
    \@ifnotempty{##2}{\nobreak\noindent\urladdrname
      \@ifnotempty{##1}{, \ignorespaces##1\unskip}\/:\space
      \ttfamily##2\par}\endgroup}%
  \addresses
  \endgroup
  \global\let\addresses=\@empty
}
\def\@setabstracta{%
  \ifvoid\abstractbox
  \else
    \skip@20pt \advance\skip@-\lastskip
    \advance\skip@-\baselineskip \vskip\skip@
    \box\abstractbox
    \prevdepth\z@ 
    \vskip-22pt
  \fi
}
\renewenvironment{abstract}{%
  \ifx\maketitle\relax
    \ClassWarning{\@classname}{Abstract should precede
      \protect\maketitle\space in AMS document classes; reported}%
  \fi
  \global\setbox\abstractbox=\vtop \bgroup
    \normalfont\small
    \list{}{\labelwidth\z@
      \leftmargin0pc \rightmargin\leftmargin
      \listparindent\normalparindent \itemindent\z@
      \parsep\z@ \@plus\p@
      
    }%
    \item[\hskip\labelsep\bfseries\abstractname.]%
}{%
  \endlist\egroup
  \ifx\@setabstract\relax \@setabstracta \fi
}

\def\ps@headings{\ps@empty
  \def\@evenhead{%
    \setTrue{runhead}%
    \normalfont\scriptsize
    \rlap{\thepage}\hfill
    \def\thanks{\protect\thanks@warning}%
    \leftmark{}{}}%
  \def\@oddhead{%
    \setTrue{runhead}%
    \normalfont\scriptsize
    \def\thanks{\protect\thanks@warning}%
    \rightmark{}{}\hfill \llap{\thepage}}%
  \let\@mkboth\markboth
}\ps@headings

\def\section{\@startsection{section}{1}%
  \z@{-1.4\linespacing\@plus-.5\linespacing}{.8\linespacing}%
  {\normalfont\bfseries\Large\raggedright}}
\def\subsection{\@startsection{subsection}{2}%
  \z@{-.8\linespacing\@plus-.3\linespacing}{.5\linespacing\@plus.2\linespacing}%
  {\normalfont\bfseries\large}}
\def\subsubsection{\@startsection{subsubsection}{3}%
  \z@{.7\linespacing\@plus.2\linespacing}{-1.5ex}%
  {\normalfont\bfseries}}
\def\@secnumfont{\bfseries}

\renewcommand\contentsnamefont{\bfseries\large}
\def\@starttoc#1#2{\begingroup
  \setTrue{#1}%
  \par\removelastskip\vskip\z@skip
  \@startsection{}\@M\z@{\linespacing\@plus\linespacing}%
    {.5\linespacing}{
      \contentsnamefont}{#2}%
  \ifx\contentsname#2%
  \else \addcontentsline{toc}{section}{#2}\fi
  \makeatletter
  \@input{\jobname.#1}%
  \if@filesw
    \@xp\newwrite\csname tf@#1\endcsname
    \immediate\@xp\openout\csname tf@#1\endcsname \jobname.#1\relax
  \fi
  \global\@nobreakfalse \endgroup
  \addvspace{32\p@\@plus14\p@}%
  \let\tableofcontents\relax
}
\def\contentsname{Contents}
\def\l@section{\@tocline{1}{.5ex}{0mm}{5pc}{}}
\def\l@subsection{\@tocline{2}{0pt}{2em}{5pc}{}}
\makeatother

\fi 

\def\to{\mathchoice{\longrightarrow}{\rightarrow}{\rightarrow}{\rightarrow}}
\makeatletter
\newcommand{\shortxra}[2][]{\ext@arrow 0359\rightarrowfill@{#1}{#2}}
\def\longrightarrowfill@{\arrowfill@\relbar\relbar\longrightarrow}
\newcommand{\longxra}[2][]{\ext@arrow 0359\longrightarrowfill@{#1}{#2}}
\renewcommand{\xrightarrow}[2][]{\mathchoice{\longxra[#1]{#2}}%
  {\shortxra[#1]{#2}}{\shortxra[#1]{#2}}{\shortxra[#1]{#2}}}
\makeatother

\makeatletter
\def\Nopagebreak{\@nobreaktrue\nopagebreak}

\makeatother

\makeatletter
\let\save@mathaccent\mathaccent
\newcommand*\if@single[3]{%
  \setbox0\hbox{${\mathaccent"0362{#1}}^H$}%
  \setbox2\hbox{${\mathaccent"0362{\kern0pt#1}}^H$}%
  \ifdim\ht0=\ht2 #3\else #2\fi
  }
\newcommand*\rel@kern[1]{\kern#1\dimexpr\macc@kerna}
\newcommand*\widebar[1]{\@ifnextchar^{{\wide@bar{#1}{0}}}{\wide@bar{#1}{1}}}
\newcommand*\wide@bar[2]{\if@single{#1}{\wide@bar@{#1}{#2}{1}}{\wide@bar@{#1}{#2}{2}}}
\newcommand*\wide@bar@[3]{%
  \begingroup
  \def\mathaccent##1##2{%
    \let\mathaccent\save@mathaccent
    \if#32 \let\macc@nucleus\first@char \fi
    \setbox\z@\hbox{$\macc@style{\macc@nucleus}_{}$}%
    \setbox\tw@\hbox{$\macc@style{\macc@nucleus}{}_{}$}%
    \dimen@\wd\tw@
    \advance\dimen@-\wd\z@
    \divide\dimen@ 3
    \@tempdima\wd\tw@
    \advance\@tempdima-\scriptspace
    \divide\@tempdima 10
    \advance\dimen@-\@tempdima
    \ifdim\dimen@>\z@ \dimen@0pt\fi
    \rel@kern{0.6}\kern-\dimen@
    \if#31
      \overline{\rel@kern{-0.6}\kern\dimen@\macc@nucleus\rel@kern{0.4}\kern\dimen@}%
      \advance\dimen@0.4\dimexpr\macc@kerna
      \let\final@kern#2%
      \ifdim\dimen@<\z@ \let\final@kern1\fi
      \if\final@kern1 \kern-\dimen@\fi
    \else
      \overline{\rel@kern{-0.6}\kern\dimen@#1}%
    \fi
  }%
  \macc@depth\@ne
  \let\math@bgroup\@empty \let\math@egroup\macc@set@skewchar
  \mathsurround\z@ \frozen@everymath{\mathgroup\macc@group\relax}%
  \macc@set@skewchar\relax
  \let\mathaccentV\macc@nested@a
  \if#31
    \macc@nested@a\relax111{#1}%
  \else
    \def\gobble@till@marker##1\endmarker{}%
    \futurelet\first@char\gobble@till@marker#1\endmarker
    \ifcat\noexpand\first@char A\else
      \def\first@char{}%
    \fi
    \macc@nested@a\relax111{\first@char}%
  \fi
  \endgroup
}
\makeatother

\newtheorem{theorem}{Theorem}[section]
\newtheorem{theoremalpha}{Theorem}

\newtheorem{corollary}[theorem]{Corollary}
\newtheorem{corollaryalpha}[theoremalpha]{Corollary}
\newtheorem{lemma}[theorem]{Lemma}

\newtheorem*{claim}{Claim}
\theoremstyle{definition}
\newtheorem{definition}[theorem]{Definition}

\newtheorem{remark}[theorem]{Remark}

\newtheoremstyle{theorem-giventitle}
        {}{}              
        {\itshape}                      
        {}                              
        {\bfseries}                     
        {.}                             
        {5pt plus 3pt minus 1pt}      
        {\thmnote{\bfseries#3}}
\theoremstyle{theorem-giventitle}
\newtheorem{theorem-named}{}
\newtheoremstyle{definition-giventitle}
        {}{}              
        {}                      
        {}                              
        {\bfseries}                     
        {.}                             
        {5pt plus 3pt minus 1pt}      
        {\thmnote{\bfseries#3}}
\theoremstyle{definition-giventitle}
\newtheorem{question-named}{}
\newtheorem{step-named}{}
\newtheorem{claim-named}{}

\numberwithin{equation}{section}

\def\Z{\mathbb{Z}}

\def\R{\mathbb{R}}

\def\cC{\mathcal{C}}
\def\cD{\mathcal{D}}

\def\cR{\mathcal{R}}

\def\bF{{\mathbb{F}}}

\def\inte{\operatorname{int}}
\def\id{\mathrm{id}}

\makeatletter
\def\tpmod#1{{\@displayfalse\pmod{#1}}}
\makeatother

\def\sm{\mathrm{sm}}
\def\top{\mathrm{top}}

\def\dax{\operatorname{dax}}
\def\Dax{\operatorname{Dax}}
\def\FQ{\operatorname{FQ}}
\def\Diff{\operatorname{Diff}}
\def\Homeo{\operatorname{Homeo}}

\def\Map{\operatorname{Map}}

\def\TOP{\mathrm{TOP}}

\def\setminus{\smallsetminus}

\def\isomto{\mathrel{\smash{\xrightarrow{\smash{\lower.4ex\hbox{$\scriptstyle{\cong}$}}}}}}
\def\isomfrom{\mathrel{\smash{\xleftarrow{\smash{\lower.4ex\hbox{$\scriptstyle{\cong}$}}}}}}

\begin{document}
\thispagestyle{empty}
\title{Light bulb smoothing for topological surfaces in 4-manifolds}

\author{Jae Choon Cha}
\address{
  Center for Research in Topology and Department of Mathematics\\
  POSTECH\\
  Pohang Gyeongbuk 37673\\
  Republic of Korea
}
\email{jccha@postech.ac.kr}

\author{Byeorhi Kim}
\address{
  Center for Research in Topology\\
  POSTECH\\
  Pohang Gyeongbuk 37673\\
  Republic of Korea
}
\email{byeorhikim@postech.ac.kr}

\begin{abstract}
  We present new smoothing techniques for topologically embedded surfaces in smooth 4-manifolds, which give topological isotopy to a smooth surface.
  As applications, we prove ``topological = smooth'' results in dimension 4 for certain disks and spheres modulo isotopy.
  A key step in our approach is to link Quinn's smoothing theory with ideas in Gabai's 4-dimensional light bulb theorem and succeeding developments of Schneiderman-Teichner and Kosanovi\'c-Teichner.
  As another application of our smoothing technique, we obtain a topological version of the Dax invariant which gives topological isotopy obstructions for topological disks in 4-manifolds.
\end{abstract}

\maketitle

\section{Introduction and main results}

A major challenge in 4-dimensional topology is to understand how the topological and smooth categories are similar and different.
There have been enormous advances since the work of Freedman and Donaldson, and especially it has been a very successful direction to find smoothly distinct but topologically equivalent objects.
Recall that the special sophistication of dimension 4 is from the difficulty of finding embedded disks, most importantly for Whitney moves.
Known differences between the topological and smooth categories can be understood as a failure of smoothing topological disks in 4-manifolds.
Topological disks from Freedman's work and subsequent developments eventually lead us to fundamental topological results such as the classification of simply connected 4-manifolds, while modern smooth techniques detect a surprisingly large extent of the failure of smoothing for many of those topological disks.

On the other hand, it is much less understood when the smooth and topological cases behave in the same way in dimension~4.
From the viewpoint of disks, the question is: when are topologically embedded disks smoothable?
Results in this direction could potentially bring us closer to affirmative answers to major open questions about the equivalence of topological and smooth categories in specific cases of interest, such as the smooth 4-dimensional Poincar\'e conjecture and (non-)existence questions for exotic smooth structures on other small 4-manifolds.

In this paper, we develop new smoothing techniques for topologically embedded disks and more generally surfaces of arbitrary genus in smooth 4-manifolds.
We focus on smoothing by topological isotopy.
That is, we study when a topologically embedded surface in a smooth 4-manifold is topologically isotopic to a smooth embedding.
Our main smoothing results are Theorems~\ref{theorem:light-bulb-smoothing} and~\ref{theorem:stable-smoothing} below in this introduction.

As an application, we prove a ``topological = smooth'' result for the isotopy classification of certain disks.
Throughout this paper, surfaces in a 4-manifold are properly embedded, and their isotopy is rel~$\partial$, i.e.,\ the boundary is fixed pointwise.
Topological embeddings and isotopy are locally flat.

\begin{corollaryalpha}
  \label{corollary:top-equal-smooth-disk-isotopy}
  Let $M$ be a smooth orientable 4-manifold.
  Consider smooth and topological disks in $M$ with a fixed smooth boundary $K\subset\partial M$.
  If a meridian of $K$ is null-homotopic in $\partial M \setminus K$, then the following natural map between smooth and topological isotopy classes is a bijection:
  \[
    \left\{
      \begin{tabular}{@{}c@{}}
        smooth disks in $M$ \\ bounded by~$K$
      \end{tabular}
    \right\}
    \Big/\,
    \begin{tabular}{@{}c@{}}
      smooth \\ isotopy
    \end{tabular}
    \;\xrightarrow{\;\approx\;}\;
    \left\{
      \begin{tabular}{@{}c@{}}
        topological disks in $M$ \\ bounded by~$K$
      \end{tabular}
    \right\}
    \Big/\,
    \begin{tabular}{@{}c@{}}
      topological \\ isotopy
    \end{tabular}
  \]
  That is, every topological disk in $M$ bounded by $K$ is topologically isotopic to a smooth disk, and two smooth disks in $M$ bounded by $K$ are topologically isotopic if and only if they are smoothly isotopic.
\end{corollaryalpha}

We remark that the condition on the meridian in Corollary~\ref{corollary:top-equal-smooth-disk-isotopy} is an algebraic property in $\pi_1(\partial M\setminus K)$ but equivalent to the existence of a geometric dual sphere of $K$ in~$\partial M$ by the loop theorem.
See Lemma~\ref{lemma:dual-in-3-manifold}.
Here, a \emph{geometric dual sphere} of a connected submanifold $R$ of codimension~2, or a \emph{geometric dual} for short, is a framed embedded sphere $G$ such that $R$ intersects $G$ at exactly one point and the intersection is transverse.

We also obtain a ``topological = smooth'' result for spheres with a common geometric dual.

\begin{corollaryalpha}
  \label{corollary:top-equal-smooth-sphere-isotopy}
  Let $M$ be a smooth orientable 4-manifold and $G$ be a smooth framed sphere in~$M$.
  Then the following natural map is a bijection:
  \begin{multline*}
    \left\{\begin{tabular}{@{}c@{}}
      smooth spheres in $M$ having\\
      $G$ as a geometric dual
    \end{tabular}\right\}
    \Big/\,
    \begin{tabular}{@{}c@{}}
      smooth
      \\
      isotopy
    \end{tabular}
    \\
    \xrightarrow{\;\;\approx\;\;}
    \left\{\begin{tabular}{@{}c@{}}
      topological spheres in $M$ having\\
      $G$ as a geometric dual
    \end{tabular}\right\}
    \Big/\,
    \begin{tabular}{@{}c@{}}
      topological
      \\
      isotopy
    \end{tabular}
  \end{multline*}
\end{corollaryalpha}

The following additional conventions apply throughout this paper.
All manifolds are oriented and 4-manifolds are connected.
Surfaces are connected and compact.
For a topological surface $R$ in a smooth 4-manifold~$M$, the restriction $\partial R \hookrightarrow \partial M$ is always isotopic to a smooth embedding, so we assume that $R$ is smooth near~$\partial R$.
Surfaces and their isotopy in a topological 4-manifold not equipped with a smooth structure mean topologically embedded surfaces and topological isotopy.

\subsection{Light bulb smoothing for surfaces}

A key idea in our approach is to link Quinn's smoothing theory for 4-manifolds~\cite{Quinn:1982-1,Quinn:1984-2,Freedman-Quinn:1990-1}, which is built on Freedman's work~\cite{Freedman:1982-1,Freedman:1984-1,Freedman-Quinn:1990-1}, with ideas in Gabai's recent work on the 4-dimensional light bulb theorem~\cite{Gabai:2020-1}.
Specifically, it leads us to the following main result.

\begin{theoremalpha}[Light bulb smoothing for surfaces]
  \label{theorem:light-bulb-smoothing}
  If a topological surface $R$ (with possibly empty boundary) in a smooth 4-manifold $M$ has a smooth geometric dual $G$, then $R$ is topologically isotopic to a smooth surface in $M$, via a topological isotopy supported in an arbitrary neighborhood of $R\cup G$.
\end{theoremalpha}

As in Corollary~\ref{corollary:top-equal-smooth-disk-isotopy}, the geometric dual condition in this statement can be reformulated to an algebraic condition: the following is equivalent to the above statement.

\begin{theorem-named}
  [An alternative form of Theorem~\ref{theorem:light-bulb-smoothing}]
  Let $R$ be a topological surface in a smooth 4-manifold~$M$.
  If a boundary component $K$ of $R$ has a meridian null-homotopic in $\partial M\setminus K$, then $R$ is topologically isotopic to a smooth surface in~$M$, via a topological isotopy supported in an arbitrary neighborhood of $R\cup \partial M$.
\end{theorem-named}

Note that the surjectivity in Corollaries~\ref{corollary:top-equal-smooth-disk-isotopy} and~\ref{corollary:top-equal-smooth-sphere-isotopy}\@ follows immediately from Theorem~\ref{theorem:light-bulb-smoothing}\@.
We also use Theorem~\ref{theorem:light-bulb-smoothing} to develop ingredients of the proof of the injectivity.

\subsection{Smoothing surfaces in a stabilization}

As another application of Theorem~\ref{theorem:light-bulb-smoothing}, we obtain a smoothing result for arbitrary topological surfaces \emph{without the need for a dual}, at the cost of $(S^2\times S^2)$-stabilizations of the ambient 4-manifold.

\begin{theoremalpha}[Stable smoothing for surfaces]
  \label{theorem:stable-smoothing}
  \sloppy
  A topological surface $R$ in a smooth 4-manifold $M$ is topologically isotopic to a smooth embedding in $M\#k(S^{2}\times S^{2})$ for some~$k$.
\end{theoremalpha}

Here, the connected sum with $k(S^2\times S^2)$ is understood as taken away from~$R$ so that $R$ can be regarded as a topological surface in $M \# k(S^2\times S^2)$.

Theorem~\ref{theorem:stable-smoothing} also holds when we allow $R$ to have multi-components.
See Section~\ref{section:proof-stable-smoothing}.

The proof of Theorem~\ref{theorem:stable-smoothing} relies on the Light Bulb Smoothing Theorem~\ref{theorem:light-bulb-smoothing} and a controlled version of Quinn's smoothing theory.
Briefly speaking, Quinn showed that a topological surface can be changed to a smooth immersion by controlled topological finger moves, i.e., finger moves with arbitrarily small diameter.
We show that when the diameter is small enough, the topological Whitney disks introduced by the finger moves are \emph{isotopic} to smooth Whitney disk after stabilization, using the Light Bulb Smoothing Theorem~\ref{theorem:light-bulb-smoothing}\@.
This enables us to smoothly undo the topological finger moves.
Details are in Section~\ref{section:proof-stable-smoothing}.

Stabilization of the ambient 4-manifold in Theorem~\ref{theorem:stable-smoothing} is necessary since there exist non-smoothable topological surfaces in smooth 4-manifolds.
For example, it is well known that there are knots $K$ in $S^3$ which are topologically slice but not smoothly slice.
That is, $K$ bounds a topological disk in $D^4$ but bounds no smooth disk in~$D^4$.
By Schneiderman~\cite{Schneiderman:2010-1}, if a knot $K$ has vanishing Arf invariant, $K$ is smoothly slice in $D^4\# k(S^2\times S^2)$ for some~$k$.
In particular, when $K$ bounds a topological slicing disk in $D^4$, $K$ bounds a smooth slicing disk in $D^4\# k(S^2\times S^2)$, but this is not isotopic to a given topological slicing disk in general.
Theorem~\ref{theorem:stable-smoothing} readily gives the following:

\begin{corollaryalpha}
  \label{corollary:stable-smoothing-slice-disk}
  Let $K$ be a topologically slice knot in $S^3$.
  Then any topological slicing disk for $K$ in $D^4$ is topologically isotopic to a smooth slicing disk in $D^4\# k(S^2\times S^2)$ for some~$k$.
\end{corollaryalpha}

Motivated by Theorem~\ref{theorem:stable-smoothing} and Corollary~\ref{corollary:stable-smoothing-slice-disk}, it appears to be intriguing to study the minimal number $k$ of ambient 4-manifold stabilization required to smoothen a given topological disk or more generally a surface.

\subsection{Topological applications of surface smoothing}

For topological manifolds, the seminal work of Kirby and Siemenmann~\cite{Kirby-Siebenmann:1977-1} establishes fundamental tools such as transversality and handle decomposition in high dimensions by developing a smoothing theory that allows us to import theorems for smooth manifolds to the topological case.
Dimension 4 is more sophisticated, but based on Freedman's work and subsequent developments, the same approach shows that fundamental theorems work, to a large extent, for topological 4-manifolds~\cite{Freedman-Quinn:1990-1}.

Our smoothing technique enables us to employ a similar approach for topological surfaces in 4-manifolds: Theorems~\ref{theorem:light-bulb-smoothing} and~\ref{theorem:stable-smoothing} isotope topological surfaces to smooth surfaces, to which we can apply smooth techniques, to import smooth results into topological cases.
It leads us to several applications.
We discuss some of them below.

\subsubsection*{Topological Dax invariant}

Recently, an invariant of Dax in~\cite{Dax:1972-1} has received significant attention, owing to work of Gabai~\cite{Gabai:2021-1}, Kosanovi\'c-Teichner~\cite{Kosanovic-Teichner:2024-1,Kosanovic-Teichner:2024-2} and Schwartz~\cite{Schwartz:2021-1}.
They used the Dax invariant of smooth disks to study 4-dimensional light bulb theorems and generalizations.
In the known approach to the Dax invariant, the smoothness of disks is essential and it appears to be difficult, if not impossible, to carry it out in the topological case (see Remark~\ref{remark:dax-why-smoothing-approach} for a related discussion).
In this paper, we use our smoothing technique to generalize the Dax invariant to topological disks.
This topological version of the Dax invariant seems to be of independent interest, and it is also an ingredient of the proof of Corollary~\ref{corollary:top-equal-smooth-disk-isotopy}\@.

\begin{theoremalpha}
  \label{theorem:dax-top}
  Let $M$ be a topological 4-manifold and $K$ be a circle embedded in~$\partial M$.
  Fix an arbitrary almost smooth structure of~$M$.
  Then there is a function
  \[
    \Dax^\top\colon \Big\{ (D_0,D_1)\,\Big|\,
    \begin{tabular}{@{}c@{}}
      $D_0$, $D_1$ are topological disks in $M$\\
      bounded by $K$, homotopic rel $\partial$ 
    \end{tabular} \Big\}
    \to \Z[\pi_1M\setminus 1]^\sigma  / d(\pi_3 (M))
  \]
  satisfying the following.
  \begin{enumerate}
    \item $\Dax^\top(D_0,D_1)$ is determined by the topological isotopy classes of $D_0$ and $D_1$.
    If $D_0$ and $D_1$ are topologically isotopic, $\Dax^\top(D_0,D_1)=0$.
    \item $\Dax^\top(D_0,D_1)$ is equal to the value of the Dax invariant defined for smooth disks in~\cite{Gabai:2021-1,Kosanovic-Teichner:2024-1,Kosanovic-Teichner:2024-2,Schwartz:2021-1} if $D_0$ and $D_1$ are smooth with respect to the almost smooth structure.
  \end{enumerate}
\end{theoremalpha}

Here, an almost smooth structure of $M$ means a smooth structure on the complement of an interior point~$p_0$.
It is known that every topological 4-manifold admits an almost smooth structure~\cite[8.2]{Freedman-Quinn:1990-1}.
When we say that a disk is smooth with respect to an almost smooth structure, we implicitly assume that the disk does not contain the singular point~$p_0$.
For a description of the range of $\Dax^\top$, see \cite{Gabai:2021-1,Kosanovic-Teichner:2024-2} or Section~\ref{subsection:preliminary-smooth-dax}.

Briefly speaking, we define $\Dax^\top(D_0,D_1)$ to be the value of the smooth Dax invariant of \emph{smoothings} of $D_0$ and $D_1$ in a stabilization of~$M$.
The Stable Smoothing Theorem~\ref{theorem:stable-smoothing} ensures the existence of smoothings of the disks.
For more details, see Sections~\ref{subsection:dax-for-top-disks} and~\ref{subsection:dax-in-top-4-manifold}.
It is noteworthy that the definition works for all pairs of homotopic disks without altering a given (almost) smooth structure of~$M$.
This is essential when we compare topological and smooth isotopy in a smooth 4-manifold.
For instance, the proof of Corollary~\ref{corollary:top-equal-smooth-disk-isotopy} benefits from this.

We remark that if a meridian of $\partial D_0$ is null-homotopic in $\partial M\setminus \partial D_0$, or equivalently if $D_0$ has a geometric dual in $\partial M$, then no stabilization is needed to define $\Dax^\top(D_0,D_1)$ for topological disks, since each $D_i$ is smoothable in $M$ by the Light Bulb Smoothing Theorem~\ref{theorem:light-bulb-smoothing}\@.

As another application of Theorem~\ref{theorem:dax-top}, we prove that there exist topologically knotted smooth 3-balls in a 4-manifold.
This generalizes Gabai's result in the smooth category~\cite[Theorems~0.8 and~5.1]{Gabai:2021-1}.

\begin{corollaryalpha}
  \label{corollary:knotted-3-disk}
  Let $M = (S^2\times D^2) \mathbin{\mathord{\#}_\partial} (S^1\times D^3)$ and $\Delta_0 = *\times D^3\subset M$ be the 3-ball in the $S^1\times D^3$ factor.
  Then there exist infinitely many smooth 3-balls properly embedded in $M$ which are homotopic to $\Delta_0$ rel $\partial$ but not pairwise topologically isotopic rel~$\partial$.
\end{corollaryalpha}

We give more applications of the topological Dax invariant in Corollaries~\ref{corollary:disk-top-isotopy-classification} and~\ref{corollary:schwartz-top}.

\subsubsection*{Topological light bulb theorems}

Our approaches and results also lead us to straightforward proofs of topological category versions of recently discovered smooth 4-dimensional light bulb theorems and related results.
As an example, in a topological 4-manifold, we obtain an isotopy classification of homotopic disks with a common geometric dual in the boundary, in terms of the Dax invariant, by combining a smooth category result of Kosonovi\'c and Teichner~\cite{Kosanovic-Teichner:2024-2} with our method.

\begin{corollaryalpha}
  \label{corollary:disk-top-isotopy-classification}
  Let $M$ be a topological 4-manifold and $D_0$ be a disk in $M$ such that a meridian of $\partial D_0$ is null-homotopic in $\partial M \setminus \partial D_0$.
  Then the topological Dax invariant gives rise to a bijection
  \[
    \Dax^\top(D_0,-)\colon
    \left\{\begin{tabular}{@{}c@{}}
      disks in $M$ homotopic\\
      to $D_0$ rel~$\partial$
    \end{tabular}\right\}
    \Big/
    \text{isotopy}
    \to
    \Z[\pi_1M\setminus 1]^\sigma / d(\pi_3M).
  \]
  In particular, $D_0$ and $D_1$ are isotopic if and only if $D_0$ and $D_1$ are homotopic rel~$\partial$ and $\Dax^\top(D_0, D_1)=0$.
\end{corollaryalpha}

The following is an analogous classification for homotopic spheres in a topological 4-manifold with a common geometric dual.
We obtain it by combining our results with work of Schneiderman and Teichner~\cite{Schneiderman-Teichner:2022-1} which generalizes the light bulb theorem of Gabai~\cite{Gabai:2020-1}.

\begin{corollaryalpha}
  \label{corollary:sphere-top-isotopy-classification}
  Let $M$ be a topological 4-manifold and $R_0$ be a sphere in $M$ with a geometric dual~$G$.
  Then there is a bijection
  \[
    \left\{\begin{tabular}{@{}c@{}}
      spheres in $M$ homotopic to $R_0$\\
      having $G$ as a geometric dual
    \end{tabular}\right\}
    \Big/
    \text{isotopy}
    \to
    \bF_2 T_M / \mu(\pi_3(M))
  \]
  induced by a topological version of the Freedman-Quinn invariant.
\end{corollaryalpha}

For more about the bijection, see Section~\ref{section:application-to-spheres}.

In Section~\ref{section:more-applications}, we discuss more applications related to light bulb theorems, which generalize previously known smooth category results to the topological category.
See Theorems~\ref{theorem:top-light-bulb-theorem} and Corollaries~\ref{corollary:top-mapping-class-S^2xD^2} and~\ref{corollary:schwartz-top}.
In particular, Corollary~\ref{corollary:schwartz-top} shows that the main result of recent work of Schwartz \cite[Theorem~0.1]{Schwartz:2021-1} holds in the topological category.

We finish this introduction with another ``topological = smooth'' result, which concerns concordance of spheres with a common geometric dual.

\begin{corollaryalpha}
  \label{corollary:top-equal-smooth-concordance-isotopy}
  Let $M$ be a smooth 4-manifold and $G$ be a smooth framed sphere in~$M$.
  Then the following natural map is a bijection.
  \begin{multline*}
    \left\{\begin{tabular}{@{}c@{}}
      smooth spheres in $M$ with\\
      $G$ as a geometric dual
    \end{tabular}\right\}
    \Big/
    \begin{tabular}{@{}c@{}}
      smooth
      \\
      concordance
    \end{tabular}
    \\
    \xrightarrow{\;\;\approx\;\;}
    \left\{\begin{tabular}{@{}c@{}}
      topological spheres in $M$ with\\
      $G$ as a geometric dual
    \end{tabular}\right\}
    \Big/
    \begin{tabular}{@{}c@{}}
      topological
      \\
      concordance
    \end{tabular}
  \end{multline*}
  Also, for smooth spheres with a common smooth geometric dual, the following are equivalent: topological concordance, topological isotopy, smooth concordance and smooth isotopy.
\end{corollaryalpha}

The smoothing problem for topological disks and surfaces is of foundational importance in 4-dimensional topology, and we expect the methods developed in this paper to have further applications.
In particular, our techniques and results (along with their further generalizations) appear to be useful for studying structures in the topological mapping class groups of non-simply-connected 4-manifolds.
We will address this elsewhere.

The organization of the body of this paper is as follows.
In Section~\ref{section:proof-light-bulb-smoothing}, we prove the Light Bulb Smoothing Theorem~\ref{theorem:light-bulb-smoothing}\@.
In Section~\ref{section:proof-stable-smoothing}, we prove the Stable Smoothing Theorem~\ref{theorem:stable-smoothing}\@.
In Section~\ref{section:top-dax-invariant}, we develop the topological Dax invariant, and in particular prove Theorem~\ref{theorem:dax-top}\@.
In Section~\ref{section:proof-top=smooth-for-disk}, we prove Corollaries~\ref{corollary:top-equal-smooth-disk-isotopy} and~\ref{corollary:disk-top-isotopy-classification}, which address isotopy classification of disks, using results developed in earlier sections.
In the same section, we also prove Corollary~\ref{corollary:knotted-3-disk}\@.
In Section~\ref{section:application-to-spheres}, we prove Corollaries~\ref{corollary:top-equal-smooth-sphere-isotopy}, \ref{corollary:sphere-top-isotopy-classification} and~\ref{corollary:top-equal-smooth-concordance-isotopy}\@.
In Section~\ref{section:more-applications}, we discuss some applications in the topological category related to light bulb theorems.
In an appendix, we provide technical details of the Freedman-Quinn invariant in the topological category, using Lees' immersion theorem~\cite{Lees:1969-1} and Kirby-Siebenmann's high dimensional smoothing theory~\cite{Kirby-Siebenmann:1977-1}, for concreteness.

\subsubsection*{Acknowledgements}
The authors thank the anonymous referees for their helpful comments.
This work was supported by the National Research Foundation grant RS-2019-NR039996.

\section{Light bulb smoothing: proof of Theorem~\ref{theorem:light-bulb-smoothing}}
\label{section:proof-light-bulb-smoothing}

Let $R$ be a topological surface in a smooth 4-manifold $M$, which has a smooth geometric dual~$G$.
Our goal is to show that $R$ is topologically isotopic to a smooth surface in~$M$, by a topological isotopy supported in a neighborhood of $R\cup G$.

By smooth approximation, $R$ is always homotopic rel $\partial$ to a smooth immersion.
But, without assuming the existence of a geometric dual, $R$ is not topologically isotopic to a smooth embedding in general.
The following deep result of Quinn~\cite{Freedman-Quinn:1990-1,Quinn:1982-1,Quinn:1984-2} shows how far one can proceed toward a smooth embedding in general.

\begin{theorem}[{\cite[8.1A]{Freedman-Quinn:1990-1}}]
  \label{theorem:quinn-finger-move}
  Let $V$ and $M$ be smooth 4-manifolds and $h\colon V \hookrightarrow M$ be a topological embedding that is smooth on the boundary.
  Let $R \subset V$ be a smooth surface.
  Then $R$ can be changed to a smoothly immersed surface $R'\subset V$ by smooth finger moves and smooth isotopy in~$V$, so that there is a topological isotopy $\{h_t\colon V\hookrightarrow M\}_{0\le t \le 1}$ of $h=h_0$ such that $h_t(V)=h(V)$ for all $t$ and $h_1\colon V\hookrightarrow M$ restricts to a smooth embedding on a neighborhood of~$R'$.
\end{theorem}

In our case, there exists a topological normal bundle $V$ for the given topological surface $R$ in $M$ by~\cite[9.3]{Freedman-Quinn:1990-1}.
It means that $V$ is a neighborhood of $R$, which has a 2-dimensional vector bundle structure over $R$ such that $R\subset V$ is identified with the zero section.
(A topological normal bundle is also required to be extendable~\cite[9.3, p.~137]{Freedman-Quinn:1980-1}, but we do not explicitly use it.)
The topological normal bundle $V$ will always be equipped with the smooth structure induced by the vector bundle structure and the (unique) smooth structure of~$R$.
Note that the inclusion $h\colon V\hookrightarrow M$ is a topological embedding but not necessarily smooth.
When $\partial R$ is nonempty, $\partial V$ is a normal bundle of the 1-manifold $\partial R$ in the 3-manifold $\partial M$, which can always be assumed to be smooth.

Quinn's Theorem~\ref{theorem:quinn-finger-move} applies directly to our inclusion $h$ and $R\subset V$, to give a smoothly immersed surface $R'\subset V$ and a topological isotopy $\{h_t\}$ of $h=h_0$.

\begin{remark}
  \label{remark:difficulty-of-undoing-finger-moves}
  The smooth finger moves in Theorem~\ref{theorem:quinn-finger-move} introduces smooth Whitney disks in~$V$ which pair up double points of~$R'$.
  Applying the isotoped embedding $h_1$, topological finger moves in $M$ changes $h_1(R)$ to $h_1(R')$ and introduces topological Whitney disks in $M$ which pair up double points of~$h_1(R')$.
  Although $h_1(R')$ is smooth in~$M$, the Whitney disks are not smooth in~$M$.
  Furthermore, in general, the Whitney disks are not smoothable in $M$, i.e.,\ not topologically isotopic to a smooth disk.
  If they were smoothable, we could use Whitney moves to undo the topological finger moves smoothly in $M$, so that we could change the smoothly immersed surface $h_1(R')$ to a smoothing of $R$ in~$M$.
  It would contradict the existence of non-smoothable surfaces in smooth 4-manifolds.
\end{remark}

In our case, the existence of a geometric dual $G$ for $R$ is essential in avoiding the issue discussed in Remark~\ref{remark:difficulty-of-undoing-finger-moves}.
We will use $G$ together with the topological Whitney disks to construct a topological isotopy of $R$ in~$M$.
(Another approach will be developed in Section~\ref{section:proof-stable-smoothing}.)

For this purpose, we first need to arrange the isotopy $\{h_t\}$ slightly more carefully.
Specifically, in addition to the properties of $R'$ and $\{h_t\}$ in Quinn's Theorem~\ref{theorem:quinn-finger-move}, we can also assume the following.
Let $\nu(G)\subset M$ be a normal bundle of~$G$.

\begin{claim}
  We may assume that the isotopy $\{h_t\}$ satisfies $h_t = h$ on $\nu(G)\cap V$ for all~$t$.
\end{claim}

Briefly, the claim holds because $R\subset M$ is already smooth near $R\cap G$ so that $R$ does not have to be modified or isotoped near $R\cap G$.
A proof can be given by applying Theorem~\ref{theorem:quinn-finger-move} in a modified way.
For concreteness, we provide technical details below.

\begin{proof}[Proof of Claim]
  Let $\{z_0\} = R\cap G$.
  Write $\nu(G)=G\times \R^2$.
  Recall that the transversality of $R$ and $G$ means that $R\cap \nu(G) = z_0\times \R^2$, the fiber over~$z_0$ (see the second paragraph of~\cite[8.3]{Freedman-Quinn:1990-1}).
  Fix an open disk neighborhood $U$ of $z_0$ in~$G$.
  The restriction of $\nu(G)$ on $U$ is equal to $U\times \R^2 \to U$.
  The other projection $U \times \R^2 = U \times (R\cap \nu(G)) \to R\cap \nu(G)$ is a partial normal bundle of $R$ on $R\cap \nu(G)$ with fiber~$U$.
  By~\cite[9.3]{Freedman-Quinn:1990-1}, if a partial normal bundle on a closed subset $C$ of $R$ extends to a neighborhood of $C$, then the given bundle on $C$ extends to a topological normal bundle on~$R$.
  In our case, we may assume that the partial normal bundle on $R\cap \nu(G)$ extends to a neighborhood of $\overline{R\cap \nu(G)}$ by shrinking $\nu(G)=G\times \R^2$ fiberwise slightly.
  So, the partial normal bundle on $R\cap \nu(G)$ extends to a topological normal bundle $V$ of $R$ in $M$ such that $V\cap \nu(G) = U \times (R\cap \nu(G))$.
  By the construction, the inclusion $h\colon V \hookrightarrow M$ is smooth on $V\cap \nu(G)$.
  Let $W = V\setminus \nu(G)$, $g=h|_{W}\colon W \hookrightarrow M$ and $Q = R\cap W = R\setminus \nu(G)$.
  Apply Quinn's Theorem~\ref{theorem:quinn-finger-move} to $(W, Q)$, to obtain a smoothly immersed surface $Q'\subset W$ by finger moves on $Q$ performed in $W$ and a topological isotopy $\{g_t\}$ of $g=g_0$ rel $\partial W=U\times \partial (\overline{R\cap \nu(G)})$ such that $g_t(W)=W$ and $g_1$ is smooth near~$Q'$.
  Now $R'=Q'\cup (R\cap \nu(G))$ and $\{h_t=g_t \cup \id_{V\cap \nu(G)}\}$ satisfy the conclusion of Theorem~\ref{theorem:quinn-finger-move} and the claim.
\end{proof}

Now, we construct an isotopy of $R$ using~$G$.
For this purpose, the following version of a smooth 4-dimensional light bulb trick due to Gabai~\cite{Gabai:2020-1} plays a key role.

\begin{lemma}[{\cite[Lemma~5.1]{Gabai:2020-1}}]
  \label{lemma:gabai-light-bulb-trick}
  In this lemma, we work in the smooth category.
  Let $R$ be a surface in a 4-manifold $N$, which has a geometric dual~$G$.
  Let $R'$ be an immersed surface in $N$ obtained from $R$ by finger moves.
  Then $R$ is isotopic, in~$N$, to a smooth surface $R''$ which lies in an arbitrarily given neighborhood of $R'\cup G$.
\end{lemma}

We remark that the surface $R''$ and the isotopy between $R$ and $R''$ in Lemma~\ref{lemma:gabai-light-bulb-trick} can be explicitly described, following~\cite[Proof of Lemma~5.1]{Gabai:2020-1}.
See Remark~\ref{remark:proof-by-picture} below. 

Returning to the proof of Theorem~\ref{theorem:light-bulb-smoothing}, let $N$ be the union of the topological normal bundle $V$ of $R$ and the normal bundle $\nu(G)$ of $G$, plumbed along $V\cap \nu(G) = U\times \R^2$.
The smooth structures on $V$ and $\nu(G)$ induce a smooth structure on~$N$.
(Note that the inclusion $f\colon N \hookrightarrow M$ is not necessarily smooth.)
Quinn's Theorem~\ref{theorem:quinn-finger-move} ensures that $R'\subset V$ has a neighborhood $U_1$ on which $h_1\colon V\hookrightarrow M$ restricts to a smooth embedding.
Apply Gabai's Lemma~\ref{lemma:gabai-light-bulb-trick} to $R$ and $R'$ in $N$, to obtain a smooth surface $R''$ in 
$U_1\cup \nu(G)$, which is smoothly isotopic to $R$ in~$N$.

The surfaces $R=f(R)$ and $R''=f(R'')$ are topologically isotopic in~$M$, since they are smoothly isotopic in~$N$ and the inclusion $f\colon N\hookrightarrow M$ is a topological embedding.
Let $f_t = h_t \cup \id_{\nu(G)}\colon N = V\cup \nu(G) \hookrightarrow M$, where $\{h_t\colon V\hookrightarrow N\}$ is the isotopy given above by Quinn's Theorem~\ref{theorem:quinn-finger-move}.
By the above claim, $f_t$ is well defined, and $\{f_t\}$ is a topological isotopy between $f=f_0$ and~$f_1$.
So the surfaces $R''=f(R'')$ and $f_1(R'')$ are topologically isotopic in~$M$.
Concatenating those two isotopies yields a topological isotopy between $R$ and $f_1(R'')$ in~$M$.

Since $h_1$ is a smooth embedding on $U_1$ and $\nu(G)$ is a smooth submanifold in~$M$, $f_1=h_1\cup \id_{\nu(G)}$ is smooth on the neighborhood $U_1\cup \nu(G)$ of~$R''$.
So $f_1(R'')$ is a smooth surface in~$M$.

Note that the topological isotopy from $R$ to $f_1(R'')$ is supported in the open set $V\cup \nu(G)$ in~$M$.
Shrinking the (topological) normal bundles $V$ and $\nu(G)$, we may assume that $V\cup \nu(G)$ is contained an arbitrarily given neighborhood of $R\cup G$.

This completes the proof of Theorem~\ref{theorem:light-bulb-smoothing}\@.

\begin{remark}
  \label{remark:proof-by-picture}
  The isotopy which gives the smoothing of the given topological surface $R$ can be depicted as follows.
  This may be viewed as a ``proof by picture.''

  \begin{figure}[t]
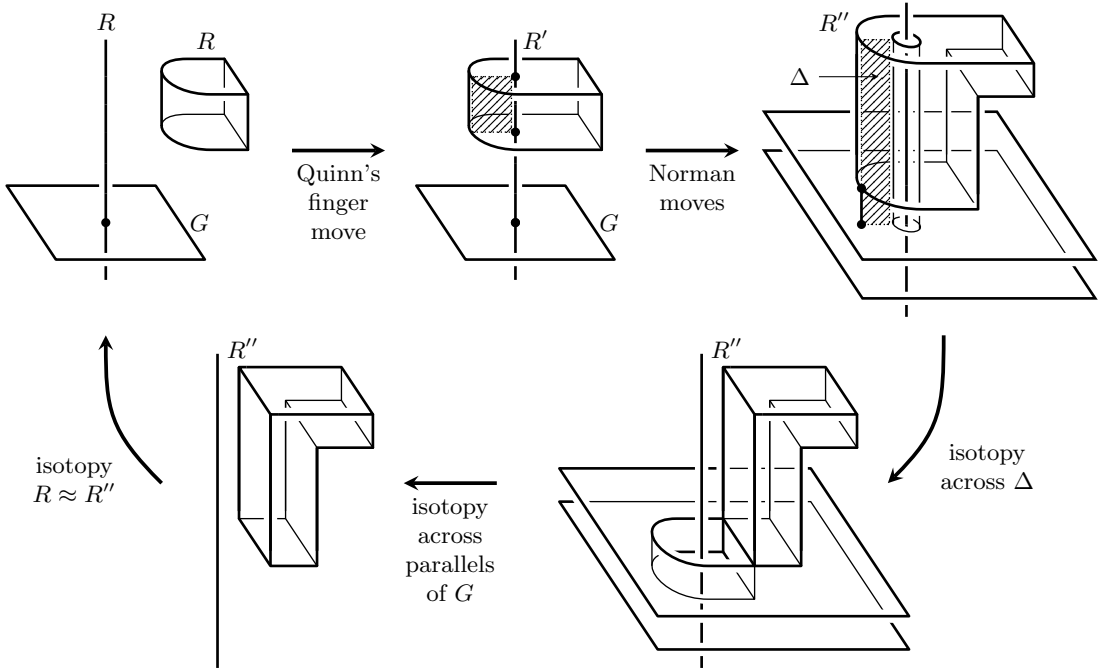

    \includestandalone{figure-light-bulb-isotopy}
    \caption{Isotopy smoothing a given topological surface~$R$.}
    \label{figure:light-bulb-isotopy-lemma}
  \end{figure}

  Figure~\ref{figure:light-bulb-isotopy-lemma} shows the case of a single finger move.
  In our case, the finger move is away from $G$, by the proof of the above claim.
  (Indeed we can always assume this since a finger move is performed along an arc which can be isotoped off~$G$.)
  Start from the top left diagram, which depicts $R$ and $G$ in $N=V\cup \nu(G)$.
  The vertical line in the diagram belongs to $R$ and extends to the past and the future.
  The next diagram shows the immersed surface $R'$ obtained by Quinn's finger move.
  The hatched rectangle is a Whitney disk introduced by the finger move.
  Now, following Gabai, apply Norman moves twice using the geometric dual~$G$, to eliminate the two double points of~$R'$.
  The third diagram shows the result, which is the surface $R''$ in Gabai's Lemma~\ref{lemma:gabai-light-bulb-trick}.
  Note that $R''$ lies in an arbitrary neighborhood of $R'\cup G$ in Figure~\ref{figure:light-bulb-isotopy-lemma}.
  So, $f_1(R'')$ is smooth in~$M$, since the embedding $f_1\colon N\hookrightarrow M$ is smooth near $R'\cup G$.

  The topological isotopy between $f_1(R)$ and~$f_1(R'')$ in~$M$ is the composition of a smooth isotopy between $R$ and $R''$ in $V$ with~$f_1$.
  The smooth isotopy is illustrated in Figure~\ref{figure:light-bulb-isotopy-lemma}.
  (See~\cite[Proof of Lemma~5.1]{Gabai:2020-1} for a description in words.)
  Start from the surface~$R''$ in the third diagram in Figure~\ref{figure:light-bulb-isotopy-lemma}.
  Note that the Whitney disk in the second diagram expands to a disk~$\Delta$ in this diagram such that $\Delta \cap R''$ is a semicircle of~$\partial\Delta$ (shown as a dotted arc).
  Apply an isotopy supported in a regular neighborhood of $\Delta$, which pushes the semicircle $\Delta \cap R''$ (rel $\partial$) to the other semicircle across~$\Delta$, to obtain the fourth diagram.
  A subsequent isotopy across (parallels of) $G$ gives the next diagram, which is readily isotopic to~$R$.

  When multiple finger moves are used to construct $R'$ from $R$, apply Figure~\ref{figure:light-bulb-isotopy-lemma} repeatedly using disjoint parallel copies of~$G$.
\end{remark}

\begin{remark}
  \label{remark:isotopy-non-smooth}
  Figure~\ref{figure:light-bulb-isotopy-lemma} also illustrates that the above isotopy from the third diagram to the fourth is non-smooth \emph{in $M$} in general.
  The disk~$\Delta$ in the third diagram, which contains the Whitney disk in the second diagram, is non-smooth in $M$ as discussed in Remark~\ref{remark:difficulty-of-undoing-finger-moves}.
  The isotopy proceeds along this non-smooth disk.
\end{remark}

\begin{remark}
  \label{remark:light-bulb-smoothing-restricted-support}
  Theorem~\ref{theorem:light-bulb-smoothing} says that the isotopy between the given topological surface $R$ and its smoothing is supported in an arbitrary neighborhood of $R\cup G$.
  By applying Theorem~\ref{theorem:light-bulb-smoothing} to $R\setminus \nu(G)$ in the exterior $M\setminus \nu(G)$ of~$G$, using a parallel $G'\subset M\setminus \nu(G)$ of $G$ as a geometric dual, we may assume that the isotopy that smoothens $R$ is supported in an arbitrary neighborhood of $R\cup G'$ and away from a neighborhood of~$G$.
  Consequently, the resulting smoothing of $R$ has $G$ as a geometric dual.
\end{remark}

\subsubsection*{Geometric dual in a 3-manifold is algebraic}
Before we finish this section, we observe the following fact mentioned in the introduction:

\begin{lemma}
  \label{lemma:dual-in-3-manifold}
  Suppose that $K$ is a locally flat embedded circle in a 3-manifold~$Y$.
  Then, $K$ has a geometric dual in $Y$ if and only if a meridian of $K$ is trivial in $\pi_1(Y \setminus K)$.
\end{lemma}

It is a straightforward consequence of the loop theorem. 
For the reader's convenience, we give a proof below.

\begin{proof}
  The only if direction is obvious.
  To verify the if direction, let $E$ be the exterior of $K$ in $Y$, and $f\colon D^2\to E$ be a null-homotopy with $f(S^1)$ a meridian of $K$ lying in~$\partial E$.
  We may assume that there is an open collar $C\cong \partial E\times [0,1)$ of $\partial E$ in $E$ such that $f^{-1}(C)$ is a collar of $\partial D^2$, $f$ is 1--1 on $f^{-1}(C)$, $f(f^{-1}(C))\cap f(D^2\setminus f^{-1}(C))=\emptyset$ and $f$ is in general position.
  Let $\Sigma(f)$ be the image of the closure of $\{x\in D^2\mid \# f^{-1}(f(x)) > 1\}$ under~$f$.
  We have $\Sigma(f)\subset f(D^2)\setminus C \subset E\setminus C$.
  Fix a neighborhood $U$ of $\Sigma(f)$ in $E$ which is disjoint from~$\partial E$.
  By the loop theorem stated in~\cite[Theorem~4.10]{Hempel:1976-1}, there is an embedding $g\colon (D^2,S^1) \hookrightarrow (E, \partial E)$ such that $g(D^2) \subset f(D^2)\cup U$.
  Since $g(S^1) \subset \partial E \cap (f(D^2)\cup U) = f(S^1)$ and $g$ is an embedding, $g(S^1)=f(S^1)$.
  That is, $g(D^2)$ is an embedded disk in $E$ bounded by a meridian $g(S^1)$ of~$K$.
  The union of $g(D^2)$ and a fiber disk in the tubular neighborhood of $K$ is a desired geometric dual for~$K$.
\end{proof}

Recall that in Corollary~\ref{corollary:top-equal-smooth-disk-isotopy}, we assume that the boundary $K$ of the concerned disk $D$ in a 4-manifold $M$ has a meridian which is null-homotopic in $\partial M\setminus K$.
By Lemma~\ref{lemma:dual-in-3-manifold}, this assumption is equivalent to that the disk $D$ has a geometric dual lying in~$\partial M$.

It is also readily seen that Theorem~\ref{theorem:light-bulb-smoothing} and its alternative form in the introduction are equivalent.
If $R$ is as in the alternative form, then $R$ has a geometric dual in $\partial M$ by Lemma~\ref{lemma:dual-in-3-manifold}, and thus the conclusion of the alternative form holds by Theorem~\ref{theorem:light-bulb-smoothing}\@.
Conversely, if $R$ is a topological surface in a smooth 4-manifold $M$ with a smooth geometric dual $G$, then in the exterior $W$ of $G$ in $M$, the alternative form of Theorem~\ref{theorem:light-bulb-smoothing} applies to the surface $R\cap W$ to conclude that the surface $R\cap W$ is smoothable by a topological isotopy in~$W$.  Consequently, $R$ is smoothable in~$M$.

\section{Smoothing in stabilizations: proof of Theorem~\ref{theorem:stable-smoothing}}
\label{section:proof-stable-smoothing}

Let $R$ be a compact 2-manifold topologically embedded in a smooth 4-manifold~$M$.
The manifold $R$~is allowed to have multi-components.
The goal is to show Theorem~\ref{theorem:stable-smoothing}, which asserts that $R$ is topologically isotopic to a smooth embedding in $M\#k(S^{2}\times S^{2})$ for some~$k$.

As in the proof of the Light Bulb Smoothing Theorem~\ref{theorem:light-bulb-smoothing}, fix a topological normal bundle $V\subset M$ of $R$ and apply Quinn's Theorem~\ref{theorem:quinn-finger-move} to each component of~$R$.
It changes $R$ to a smoothly immersed 2-manifold $R'\subset V$ by finger moves in~$V$ and gives a topological isotopy of the inclusion $h_0\colon V\hookrightarrow M$ to another topological embedding $h_1\colon V\hookrightarrow M$ which is smooth near~$R'$.
Let $D_i\subset V$ be the Whitney disks for $R'$ introduced by the smooth finger moves.
(These are the hatched rectangle in the second diagram in Figure~\ref{figure:light-bulb-isotopy-lemma}.)
In $M$, $R=h_0(R)$ is topologically isotopic to $h_1(R)$, topological finger moves change $h_1(R)$ to a smoothly immersed 2-manifold $R^+ := h_1(R')$, and $W_i := h_1(D_i) \subset M$ are topological Whitney disks for $R^+$ introduced by the topological finger moves.

The remaining part of the proof will proceed as follows.
In Remark~\ref{remark:difficulty-of-undoing-finger-moves}, we observed that the disks $W_i$ are not smoothable in $M$ in general.
In a stabilization $M \# k(S^2\times S^2)$, however, we will show that the disks $W_i$ are topologically isotopic to smooth disks~$W_i'$, by using the Light Bulb Smoothing Theorem~\ref{theorem:light-bulb-smoothing}\@.
Then, a desired smoothing of $R$ in $M \# k(S^2\times S^2)$ can be obtained from $R^+$ by Whitney moves using~$W_i'$, which undo the topological finger moves smoothly.
To do this, we need a geometric dual for~$W_i$.
It turns out that the existence of a geometric dual for $W_i$ is closely related to the double point loops associated with the finger moves.
In Section~\ref{subsection:double-point-loop-approx-embedding}, we investigate homotopy classes of double point loops.
In Section~\ref{subsection:dual-for-Whitney-disks}, we construct geometric dual spheres for the Whitney disks $W_i$ using the results of Section~\ref{subsection:double-point-loop-approx-embedding}.

\begin{remark}
  In a stabilization of $M$, a well-known application of the Norman trick changes a smooth immersion of a surface to a smooth embedding:
  tube the immersed surface to the sphere $S^2\times *$ in a $S^2\times S^2$ summand, so that $*\times S^2$ becomes a geometric dual, and then remove double points by tubing to parallels of the geometric dual (e.g.\ see~\cite[1.9]{Freedman-Quinn:1990-1}).
  In our case, one could apply this to the immersed surface $R^+$ or a smooth approximation of $W_i$, but it does not give an isotopy of the initial surface~$R$.
\end{remark}

\subsection{Double point loops of an approximation of an embedding}
\label{subsection:double-point-loop-approx-embedding}

We begin by recalling the definition of a double point loop.
For an arbitrary map $g\colon K\to V$, a \emph{double point} means a point $p\in g(K)$ such that $p=g(x)=g(y)$ for some $x$, $y\in K$, $x\ne y$.
If $\gamma$ is a path from $x$ to $y$ in $K$, the loop $g(\gamma)$ on $g(K)$ is called a \emph{double point loop} for~$p$.
In this paper, $K$ and $V$ will always be manifolds, but the map $g$ does not have to be generic to define a double point.

The following simple principle turns out to be very useful for our purpose:
``an approximation close to an embedding has null-homotopic double point loops.''
This means that if a map $g$ approximates an embedding $f$ sufficiently closely, then double points of $g$ have double point loops which are sufficiently short, and thus null-homotopic in a manifold.
The lemma below is a formal statement that we will use.
It is slightly more general than what we need in this section, for later use in this paper.
For brevity, for two maps $f$, $g\colon K\to V$ to a metric space $(V,d)$, we say that $g$ is \emph{$\epsilon$-close} to $f$ if $d(g(x),f(x))<\epsilon$ for all $x\in K$.

\begin{lemma}
  \label{lemma:pi_1-null-approx}
  Let $K$ be a compact connected smooth manifold and $V$ be a topological manifold equipped with a metric~$d$.
  The manifolds $K$ and $V$ are allowed to have nonempty boundary.
  Let $f\colon K\hookrightarrow V$ be a topological embedding.
  Then there exists $\epsilon>0$ such that if a map $g\colon K\to V$ is $\epsilon$-closed to $f$, every double point of $g$ has a double point loop which is null-homotopic in~$V$.
\end{lemma}

\begin{remark}
  In addition, if $K$ is simply connected and $g$ is an immersion, then it follows that \emph{the image $g(K)$ is $\pi_1$-null in $M$}, i.e.,\ $\pi_1(g(K)) \to \pi_1(M)$ is trivial for any choice of a basepoint.
\end{remark}

The proof of Lemma~\ref{lemma:pi_1-null-approx} is elementary.
For the reader's convenience, we spell out details below.

\begin{proof}[Proof of Lemma~\ref{lemma:pi_1-null-approx}]
  Fix a metric $\rho\colon K\times K \to [0,\infty)$ induced by a Riemannian metric on~$K$.
  In $K$ and $V$, denote the metric ball with center $x$ and radius $r$ by $B(x,r)$.
  Use the compactness of $K$ to choose constants $r>0$ and $\delta>0$ such that 
  \begin{enumerate}[label=(\roman*)]
    \item\label{item:small-ball-contractible} $B(f(x),r)$ is contained in an contractible open set in~$V$ for all $x\in K$;
    \item\label{item:short-geodesic} there is a short geodesic between $x$ and $y$ in $K$ whenever $\rho(x,y)<\delta$.
  \end{enumerate}
  In addition, since $f$ is continuous, we may assume that
  \begin{enumerate}[resume*]
    \item\label{item:control-f} $d(f(x),f(y)) < r/2$ if $\rho(x,y)<\delta$.
  \end{enumerate}
  Use the continuity of $f^{-1}\colon f(K) \to K$ to choose $\delta' >0$ such that 
  \begin{enumerate}[resume*]
    \item\label{item:control-f-inverse} $\rho(x,y)<\delta$ if $d(f(x),f(y))<\delta'$.
  \end{enumerate}

  Let $\epsilon=\min(\delta', r)/2$.
  Suppose $g \colon K\to V$ is $\epsilon$-close to~$f$.
  If $g(x)=g(y)$, then
  \[
    d(f(x), f(y))\leq d(f(x),g(x))+d(g(y), f(y))<2\epsilon \le \delta'
  \]
  and hence $\rho(x,y)<\delta$ by~\ref{item:control-f-inverse}.
  There is a short geodesic $\alpha$ between $x$ and $y$ in~$K$ by~\ref{item:short-geodesic}.
  For any $z$ on $\alpha$, $\rho(x,z)\leq \rho(x,y)<\delta$, so $d(f(x),f(z)) < r/2$ by~\ref{item:control-f}.
  It follows that
  \[
    d(f(x),g(z)) \leq d(f(x),f(z))+d(f(z),g(z) )< r/2 +\epsilon \leq r.
  \]
  That is, $g(\alpha)\subset B(f(x), r)$.
  By~\ref{item:small-ball-contractible}, the double point loop $g(\alpha)$ is null-homotopic in~$V$.
\end{proof}

Returning to the proof of Theorem~\ref{theorem:stable-smoothing}, recall that 
the smoothly immersed surface $R^+$ has topological Whitney disks $W_i$ introduced by topological finger moves in~$M$.
Let $p_i$ be one of the two double points of $R^+$ paired by~$W_i$.

\begin{lemma}
  \label{lemma:finger-move-double-point-loops}
  The topological finger moves can be arranged so that the following hold.
  \begin{enumerate}
    \item\label{item:finger-move-whitney-disk} The Whitney disks $W_i$ are pairwise disjoint and the interior of each $W_i$ is disjoint from~$R^+$.
    \item\label{item:finger-move-double-point-loop} Each double point $p_i$ of $R^+$ has a double point loop $\alpha_i\subset R^+$ such that $\alpha_i\cap W_i=\{p_i\}$, $\alpha_i\cap W_j=\alpha_i\cap \alpha_j = \emptyset$ for $i\ne j$, and $\alpha_i$ has no self-intersection.
    \item\label{item:finger-move-parallel} Each $\alpha_i$ has a parallel $\alpha_i'$ obtained by pushing $\alpha_i$ off $R^+$ slightly, avoiding $\bigsqcup_j W_j$, such that $\alpha_i'$ is null-homotopic in~$M\setminus (R^+ \cup \bigsqcup_j W_j)$.
  \end{enumerate}
\end{lemma}

Lemma~\ref{lemma:finger-move-double-point-loops}~\ref{item:finger-move-whitney-disk} and~\ref{item:finger-move-double-point-loop} are straightforward, as shown in the proof below.
Lemma~\ref{lemma:finger-move-double-point-loops}~\ref{item:finger-move-parallel} is nontrivial.
Note that when $R$ is not simply connected, $[\alpha_i]\in \pi_1(M)$ depends on the choice of~$\alpha_i$.
In addition, even when $\alpha_i$ is fixed, $[\alpha_i']\in \pi_1(M\setminus R^+)$ is not uniquely determined.
It is altered by twisting the framing that we use to push~$\alpha_i$.
So, we need to choose both $\alpha_i$ and $\alpha_i'$ carefully.

The proof of Lemma~\ref{lemma:finger-move-double-point-loops}~\ref{item:finger-move-parallel} relies on a ``controlled'' version of Theorem~\ref{theorem:quinn-finger-move}, which is a deep result of Quinn.
It essentially tells us that the finger moves can be assumed to be arbitrarily ``small.''
In our case, if finger moves were small enough, then the immersed surface $R'$ would be an approximation sufficiently close to the embedded surface $R$, so double point loops would be inessential by Lemma~\ref{lemma:pi_1-null-approx}.
For a rigorous proof, we will use an explicit statement given as Addendum to Theorem~\ref{theorem:quinn-finger-move} below, which is proven in~\cite[8.1A, 7.2C and 7.1D]{Freedman-Quinn:1990-1}\@.
We say that a homotopy $\{f_t\colon A\to X\}$ in a metric space $(X,d)$ \emph{has diameter less than $\epsilon$} if $d(f_0(x),f_t(x))<\epsilon$ for all $t$ and~$x$.
As a special case, a regular homotopy of diameter less than $\epsilon$ is defined too.

\begin{theorem-named}[Addendum to Theorem~\ref{theorem:quinn-finger-move}]
  Fix a smooth metric $d$ on~$V$.
  Let $\epsilon>0$ be arbitrarily given.
  Then in Theorem~\ref{theorem:quinn-finger-move}, we may assume that
  the smoothly immersed surface $R'\subset V$ is obtained from $R$ by applying a smooth regular homotopy of diameter less than $\epsilon$ in $V$, which consists of smooth isotopy and smooth finger moves.
\end{theorem-named}

\begin{proof}
  The statement of Theorem~\ref{theorem:quinn-finger-move} is from \cite[8.1A]{Freedman-Quinn:1990-1}, but \cite[8.1A]{Freedman-Quinn:1990-1} does not state the above addendum explicitly.
  So, for the reader's convenience, we indicate from where it can be found in~\cite{Freedman-Quinn:1990-1}.
  In the proof of \cite[Theorem~8.1A]{Freedman-Quinn:1990-1}, which is given in~\cite[8.1, p.~115--116]{Freedman-Quinn:1990-1}, a ``technically controlled $h$-cobordism theorem'' \cite[7.2C]{Freedman-Quinn:1990-1} is used.
  In particular, in the last paragraph of the proof of \cite[7.2C]{Freedman-Quinn:1990-1}, arguments in~\cite[7.1D]{Freedman-Quinn:1990-1} are used to construct the desired smooth isotopy and finger moves which changes our $R$ to $R'$ in~$V$.
  In the arguments of \cite[7.1D]{Freedman-Quinn:1990-1}, on page 107,
  the isotopy and finger moves are performed near certain immersed spheres and disks, whose sizes are small enough in case of the proof of \cite[7.2C]{Freedman-Quinn:1990-1}, so that it is controlled as desired.
  In fact, in the last sentence of the statement of \cite[Theorem~7.2C]{Freedman-Quinn:1990-1} on p.~110, it is explicitly said that the regular homotopy consists of isotopy and finger moves and has diameter less than~$\epsilon$.
\end{proof}

We remark that the isotopy $\{h_t\colon V\hookrightarrow M\}$ in Theorem~\ref{theorem:quinn-finger-move} can also be controlled to have diameter less than $\epsilon$, as stated in~\cite[8.1A]{Freedman-Quinn:1990-1}.
We do not use it in this paper.

\begin{proof}[Proof of Lemma~\ref{lemma:finger-move-double-point-loops}]
  Since the topological embedding $h_1\colon V\hookrightarrow M$ takes $R'$ and $D_i$ to $R^+$ and $W_i$, it suffices to show that Lemma~\ref{lemma:finger-move-double-point-loops}~\ref{item:finger-move-whitney-disk}, \ref{item:finger-move-double-point-loop} and~\ref{item:finger-move-parallel} hold for $(M,R^+,\{W_i\})= (V,R',\{D_i\})$.

  Let $f\colon R\hookrightarrow V$ be the inclusion.
  Apply Lemma~\ref{lemma:pi_1-null-approx} to $f$ to choose $\epsilon>0$ such that if $g\colon R\to V$ is $\epsilon$-close to $f$, then each double point of $g$ has a double point loop null-homotopic in~$V$.

  By the Addendum to Theorem~\ref{theorem:quinn-finger-move}, we may assume that there is a smooth regular homotopy $\{f_t\colon R \to V\}_{0\le t\le 1}$ with $f=f_0$, of diameter less than $\epsilon$, consisting of smooth isotopy and smooth finger moves, which changes $R=f(R)$ to $R'=f_1(R)$.

  If the $i$th finger move is performed along a guiding arc $\gamma_i$, the previously introduced Whitney disks $D_1$, $\ldots\,$,~$D_{i-1}$ can be slightly isotoped to avoid~$\gamma_i$. (This isotopy is allowed to move $\partial D_i$ in $R^+$ but fixes the double points.)
  So, we can assume that the disks $D_i$ are pairwise disjoint and the interior of each $D_i$ is disjoint from~$R'$.
  Since $f_1$ is $\epsilon$-close to the embedding $f\colon R\hookrightarrow V$, each double point $p_i$ on $R'$ has a double point loop $\alpha_i\subset R'$ which is null-homotopic in~$V$.
  Let $q_i$ be the double point paired with $p_i$ by~$W_i$.
  If some $\alpha_j$ intersects $\alpha_i$ or $\partial D_i$, apply a finger move to $\alpha_j$ on~$R^+$, along a subarc in $\partial D_i\cup \alpha_i$ from the intersection to $q_i$, to eliminate the intersection.
  See Figure~\ref{figure:clean-double-point-loop}.
  Self-intersections of $\alpha_i$ are eliminated similarly.
  So we may assume that $\alpha_i\cap D_i=\{p_i\}$, $(\alpha_i\cup D_i)\cap D_j=\emptyset$ for $i\ne j$ and $\alpha_i$ has no self intersection.
  Push $\alpha_i$ slightly off $R'$ to obtain a parallel~$\alpha_i'\subset V$ which is disjoint from $R'\cup\bigsqcup_j D_j$.

  \begin{figure}
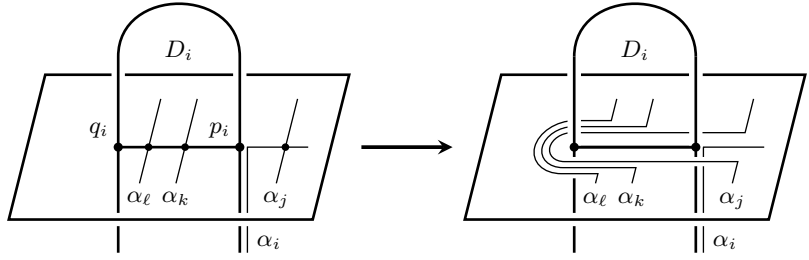

    \includestandalone{figure-clean-double-point-loop}
    \caption{Finger moves removing intersections of $\alpha_i \cup D_i$ with other $\alpha_j$.}
    \label{figure:clean-double-point-loop}
  \end{figure}

  By the fiber exact sequence
  \[
    \pi_1(\R^2\setminus 0) \to \pi_1(V\setminus R) \to \pi_1(R)
  \]
  for the $(R^2\setminus 0)$-bundle $V\setminus R \to R$, the kernel of $\pi_1(V\setminus R)\to \pi_1(R)\cong\pi_1(V)$ is generated by a meridian of~$R$.
  It is known that the complements $V\setminus (R'\cup \bigsqcup_j D_j)$ and $V\setminus R$ have isomorphic fundamental groups~\cite[Lemma~2.6; see also Figure~1]{Cha-Kim:2016-1}.
  More specifically, viewing $R$ as a subset of a regular neighborhood $N(R'\cup \bigsqcup_j D_j)$ of the 2-complex $R'\cup \bigsqcup_j D_j$, the inclusion induces an isomorphism
  \[
    \pi_1(V\setminus (R'\cup \bigsqcup_j D_j)) \xrightarrow{\cong} \pi_1(V\setminus R).
  \]
  Since inclusions commute, it follows that the kernel of $\pi_1(V\setminus (R'\cup \bigsqcup D_i)) \to \pi_1(V)$ is equal to that of $\pi_1(V\setminus R)\to \pi_1(V)$.
  It is a cyclic group generated by a meridian of~$R$.

  Since $\alpha_i$ is null-homotopic in $V$, so is $\alpha_i'$, and thus $\alpha_i'$ represents an element in the kernel of $\pi_1(V\setminus (R'\cup \bigsqcup D_i)) \to \pi_1(V)$.
  It follows that $\alpha_i'$ is the $k$th power of a meridian in $\pi_1(V \setminus (R'\cup \bigsqcup D_i))$ for some~$k$.
  Therefore, by twisting the framing used to push $\alpha_i$ to $\alpha_i'$ $-k$ times, we obtain a new parallel $\alpha_i'$ which is null-homotopic in~$V \setminus (R'\cup \bigsqcup D_i)$.
  This completes the proof of Lemma~\ref{lemma:finger-move-double-point-loops}.
\end{proof}

\begin{remark}
  When each component of the given 2-manifold $R$ is a disk or sphere, the proof of Lemma~\ref{lemma:finger-move-double-point-loops} simplifies significantly.
  In this special case, $\pi_1(V)\cong \pi_1(R)$ is trivial, so without relying on the Addendum of Theorem~\ref{theorem:quinn-finger-move}, one sees that a parallel $\alpha'$ is always a power of a meridian of~$R$ by the argument in the second paragraph from the last in the above proof.
\end{remark}

\subsection{Geometric duals for Whitney disks in a stabilization}
\label{subsection:dual-for-Whitney-disks}

In this subsection, we finish the proof of Theorem~\ref{theorem:stable-smoothing}\@.

The following result relates the $\pi_1$-nullity of $\alpha_i'$ to the existence of a \emph{collection} of \emph{immersed} spheres which are geometrically dual to the \emph{collection} of the Whitney disks~$W_i$.

\begin{lemma}[{\cite[Sections 4.2 and 4.7]{Cha-Orr-Powell:2020-1}}]
  \label{lemma:duals-from-accessory-disks}
  If each $p_i$ has a parallel $\alpha_i'$ of a double point loop which is null-homotopic in $M\setminus (R^+ \cup \bigsqcup_j W_j)$, then there exist immersed framed spheres $S_i$ in $M\setminus R^+$ such that $S_i$ is geometrically dual to $W_i$ and disjoint from $W_j$ for $i\ne j$.
\end{lemma}

In \cite[Sections 4.2 and 4.7]{Cha-Orr-Powell:2020-1}, they need more than Lemma~\ref{lemma:duals-from-accessory-disks} and thus the arguments are more complicated than what we need.
We give a simplified proof of Lemma~\ref{lemma:duals-from-accessory-disks} below.

\begin{figure}[t]
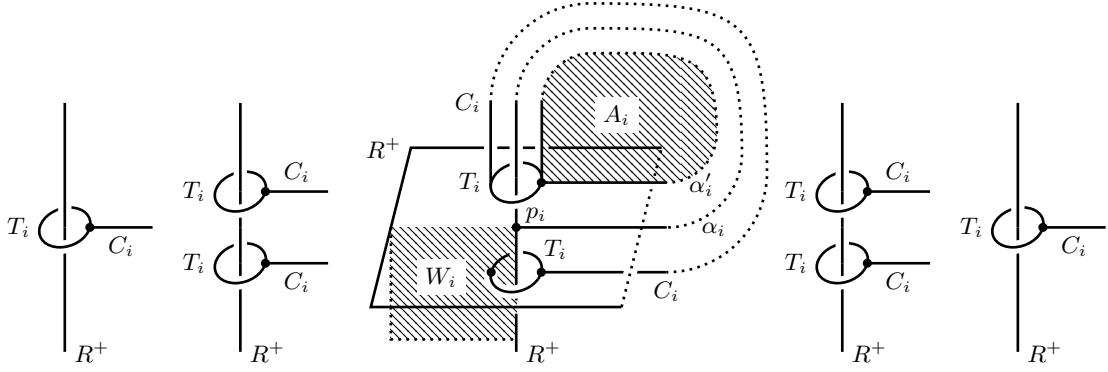

  \includestandalone{figure-double-point-nbhd}
  \caption{The double point loop $\alpha_i$, parallel $\alpha_i'$, Whitney disk $W_i$, immersed disk $A_i$ and linking torus $T_i$ near~$p_i$.}
  \label{figure:double-point-nbhd}
\end{figure}

\begin{figure}[t]
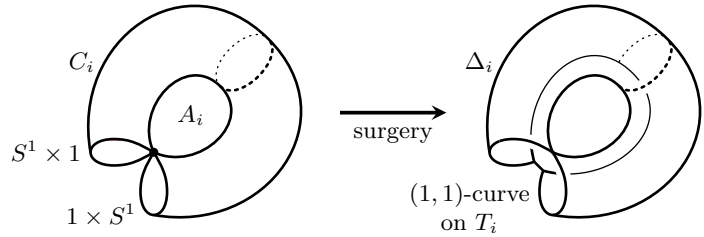

  \includestandalone{figure-surgery-on-annulus}
  \caption{Surgery on the annulus $C_i$ using the disk $A_i$.}
  \label{figure:surgery-on-annulus}
\end{figure}

\begin{proof}[Proof of Lemma~\ref{lemma:duals-from-accessory-disks}]
  Identify a smooth ball neighborhood of $p_i$ with $D^2\times D^2$ by a diffeomorphism, so that $p_i=(0,0)$, the involved sheets are $D^2\times 0$ and $0\times D^2$, $W_i\cap (D^2\times D^2) = [-1,0] \times [-1,0]$ and $\alpha_i\cap (D^2\times D^2)=([0,1]\times 0) \cup  (0\times [0,1])$.
  See Figure~\ref{figure:double-point-nbhd}.
  Note that $T_i=S^1\times S^1 \subset D^2\times D^2$ is a linking torus for~$p_i$, which is geometrically dual to~$W_i$.

  Fix a normal circle bundle $\zeta$ of $R^+$ in $M$ such that its restriction on $\alpha_i \cap (D^2\times D^2)$ is $([0,1]\times S^1) \cup (S^1\times[0,1])$.
  The restriction of $\zeta$ on $\alpha_i \setminus \inte(D^2\times D^2)$ is an annulus $C_i$ bounded by the standard basis curves $(S^1\times 1) \cup (1\times S^1)$ of~$T_i$.
  (More precisely, $C_i$ is an annulus with two boundary points identified.)
  We may assume that the parallel $\alpha_i'$ is based at $(1,1)\in S^1\times S^1$ and lies on~$C_i$.
  Let $A_i$ be an immersed disk in $M\setminus (R^+ \cup \bigsqcup W_j)$ bounded by~$\alpha_i'$.
  Perform surgery on $C_i$ along $\alpha_i'$ using two copies of $A_i$, to get a map $D^2 \to M\setminus (R^+ \cup \bigsqcup W_j)$, and perturb it slightly to obtain an immersed disk $\Delta_i$ in $M\setminus (R^+ \cup \bigsqcup W_j)$ bounded by a $(1,1)$-curve of~$T_i$.
  See Figure~\ref{figure:surgery-on-annulus}.
  The disk $\Delta_i$ has a unique framing in $M$, and we may assume that its restriction on $\partial \Delta_i$ agrees with the framing determined by the normal direction of $\partial \Delta_i$ in~$T_i$, by boundary twisting~\cite[1.3]{Freedman-Quinn:1990-1}.
  Then, using $\Delta_i$ and a parallel copy, do surgery on $T_i$ along $\partial \Delta_i$, to obtain an immersed sphere~$S_i$.
  Since $T_i$ is geometrically dual to $W_i$ and $\Delta_i \subset M\setminus (R^+ \cup \bigsqcup W_j)$, $S_i$ is geometrically dual to $W_i$ and disjoint from $W_j$ for $i \ne j$.
  We may assume that $S_i$ has vanishing self-intersection number by interior twisting~\cite[1.3]{Freedman-Quinn:1990-1}.
  (In fact, one can verify that no interior twist is needed.)
  Since $T_i$ is framed and $S_i$ is homologous to $T_i$ in~$M$, $[S_i]\cdot[S_i] = [T_i]\cdot[T_i]=0$.
  It follows that $S_i$ has vanishing Euler number, i.e.,\ $S_i$ is framed.
\end{proof}

Now we continue the proof of Theorem~\ref{theorem:stable-smoothing}\@.
In our case, each $\alpha_i'$ is null-homotopic in $M\setminus (R^+\cup\bigsqcup W_j)$ by Lemma~\ref{lemma:finger-move-double-point-loops}, so it follows that there are immersed sphere $S_i$ described in Lemma~\ref{lemma:duals-from-accessory-disks}.
By smooth approximation, we may assume that each $S_i$ is smoothly immersed and $S_i$ and $S_j$ are transverse for $i\ne j$.

Take the connected sum of $M$ with $k$ copies of $S^2\times S^2$, at an interior point of $M$ away from $R^+$, $W_i$ and $S_i$, where $k$ is the number of the Whitney disks~$W_i$.
Modify the spheres $S_i$ in $M\# k(S^2\times S^2)$ as follows.
For brevity, we denote spheres obtained by modifying $S_i$ by the same symbol~$S_i$.
Tube $S_i$ to the $S^2\times *$ factor of the $i$th copy of $S^2\times S^2$ in $M\# k(S^2\times S^2)$.
Now the sphere $*\times S^2$ in the $i$th copy of $S^2\times S^2$ is a smooth geometric dual for the spheres~$S_i$.
Remove all self and mutual intersections of the spheres $S_i$ by Norman moves, i.e.,\ for each intersection, add a parallel copy of the geometric dual of $S_i$ to one of the involved sheets, to remove the intersection.
Now each $S_i$ is a smooth geometric dual for~$W_i$, and $S_i\cup W_i$ is disjoint from $S_j\cup W_j$ for $i\ne j$.

Apply the Light Bulb Smoothing Theorem~\ref{theorem:light-bulb-smoothing}, to obtain new smooth Whitney disks $W_i'$ which are topologically isotopic to $W_i$ (rel $\partial$) in $M\# k(S^2\times S^2)$.
Since the 2-complexes $W_i\cup S_i$ are pairwise disjoint, the smooth disks $W_i'$ are pairwise disjoint and the isotopies between $W_i$ and $W_i'$ are pairwise disjoint.
Thus, $\bigsqcup W_i$ is topologically isotopic to $\bigsqcup W_i'$.

Eliminate all double points of $R^+$ by Whitney moves using~$W_i'$, and call the result~$R''\subset M\# k(S^2\times S^2)$.
The surface $R''$ is smooth in $M\# k(S^2\times S^2)$, since the immersed surface $R^+$ and the Whitney disks $W_i'$ are smooth.
Moreover, since $R$ is obtained from $R^+$ by Whitney moves using $\bigsqcup W_i$ which is topologically isotopic to $\bigsqcup W_i'$, it follows that $R$ is topologically isotopic to $R^+$ in $M\# k(S^2\times S^2)$.

This completes the proof of the Stable Smoothing Theorem~\ref{theorem:stable-smoothing}\@.

\section{Topological Dax invariant}
\label{section:top-dax-invariant}

In this section, we develop the Dax invariant for topological disks, by extending the Dax invariant of smooth disks which was used in recent studies of light bulb theorems~\cite{Gabai:2021-1,Kosanovic-Teichner:2024-1,Kosanovic-Teichner:2024-2,Schwartz:2021-1}.
In particular, we show that the Dax invariant for topological disks is invariant under topological isotopy, as stated in Theorem~\ref{theorem:dax-top} in the introduction.

In Section~\ref{subsection:preliminary-smooth-dax}, we review necessary definitions and properties of the smooth Dax invariant.
In Section~\ref{subsection:dax-for-top-disks}, we define and study the Dax invariant for topological disks in smooth 4-manifolds.
A key ingredient is our smoothing technique for disks.
In Section~\ref{subsection:dax-in-top-4-manifold}, we generalize the Dax invariant to disks in topological 4-manifolds.

\subsection{A quick review of the smooth Dax invariant}
\label{subsection:preliminary-smooth-dax}

Definitions and results in this subsection are due to Dax~\cite{Dax:1972-1}, Gabai~\cite{Gabai:2021-1} and Kosanovi\'c--Teichner~\cite{Kosanovic-Teichner:2024-1,Kosanovic-Teichner:2024-2}.
Our treatment follows closely~\cite{Gabai:2021-1,Kosanovic-Teichner:2024-2}.

We begin with a general setup for the definition of the smooth Dax invariant.
Let $M$ be a smooth 4-manifold, possibly with nonempty boundary.
Fix two points $x_0$, $x_1$ in~$M$, $x_0\ne x_1$.
Let $H\colon I^3 \to M$ be a map that satisfies the following boundary condition:
\begin{enumerate}[label=(D)]
  \item\label{item:dax-boundary-condition} $H$ is smooth on $\partial I^3$, and for each $(t,u)\in \partial I^2$, the path $\alpha_{t,u}\colon I\to M$ defined by $\alpha_{t,u}(s) = H(s,t,u)$ is an embedding such that $\alpha_{t,u}(i) = x_i$ for $i=0,1$.
\end{enumerate}
Define the \emph{honest track} $\widebar H\colon I^3\to M\times I^2$ of $H$ by $\widebar H(s,t,u)=(H(s,t,u),t,u)$.
By~\cite[Chapters III and VI]{Dax:1972-1}, $H$ can be perturbed, by homotopy rel $\partial I^3$ with arbitrarily small diameter, to a smooth map $F\colon I^3 \to M$ such that $\widebar{F} \colon I^3\to M\times I^2$ is a smooth generic immersion, i.e., $\widebar F$ is an embedding except at finitely many transverse double points.
We call $\widebar F$ a \emph{generic track} of~$H$.

A double point $p$ of $\widebar{F}$ is of the form $p=\widebar{F}(x)=\widebar{F}(y)$ with $x=(s_0,t,u)$ and $y=(s_1,t,u)$.
We may assume $s_0 < s_1$.
The boundary condition~\ref{item:dax-boundary-condition} ensures that $x$ and $y$ are interior points of~$I^3$.
The sign $\epsilon(p)\in \{-1,1\}$ of $p$ is defined by comparing the two orientations of $\widebar{F}_*(T_x(I^3)) \oplus \widebar{F}_*(T_y(I^3)) = T_p(M\times I^2)$ induced by the orientations of $I^3$ and $M\times I^2$.
The key is that the first coordinates $s_0$, $s_1$ give an order of the pre-images $x$ and $y$ of~$p$.
It enables us to define $\epsilon(p)$ without ambiguity.
In addition, let $g(p)\in \pi_1(M\times I^2) = \pi_1(M)$ be the class of the double point loop $\widebar{F}|_{[0,s_{1}]\times s\times t}\cdot(\widebar{F}|_{[0,s_{2}]\times s\times t})^{-1}$.

\begin{definition}[\cite{Dax:1972-1,Gabai:2021-1,Kosanovic-Teichner:2024-1,Kosanovic-Teichner:2024-2}]
  \label{definition:smooth-dax-for-homotopy}
  The Dax invariant of a map $H$ satisfying~\ref{item:dax-boundary-condition} is defined to be the following sum over all double points $p$ with nontrivial~$g(p)$:
  \[
    \dax(H):=\sum_{g(p)\neq 1}\epsilon(p)\cdot g(p) \in \Z[\pi_1(M) \setminus 1].
  \]
\end{definition}

Here, $\Z[\pi_1(M)\setminus 1]$ denotes the free abelian subgroup of the integral group ring $\Z[\pi_1(M)]$ generated by non-identity elements of~$\pi_1(M)$.

The invariant $\dax(H)$ has the following properties.
\begin{enumerate}
  \item \emph{Homotopy invariance.}
  The value of $\dax(H)$ is well-defined, independent of the choice of $F$, and invariant under homotopy of~$H$ rel~$\partial I^3$.
  \item \emph{Additivity.}
  If $H$, $H'\colon I^3\to M$ satisfy $H(s,t,1)=H'(s,t,0)$, 
  the concatenation $H\cdot H'\colon I^3 \to M$ is defined as usual, by $(H\cdot H')(s,t,u)=H(s,t,2u)$ for $1\le u\le \frac12$, $H'(s,t,2u-1)$ for $\frac12\le u\le 1$.
  If $H$ and $H'$ satisfy \ref{item:dax-boundary-condition}, then after altering $H$ near $u=1$ and $H'$ near $u=0$ by small homotopy if necessary, $H\cdot H'$ satisfies \ref{item:dax-boundary-condition} too, and $\dax(H\cdot H') = \dax(H) + \dax(H')$.
  \item \emph{Inversion.}
  For the homotopy inverse $H^{-1}$ defined by $H^{-1}(s,t,u)=H(s,t,1-u)$, $\dax(H^{-1}) = -\dax(H)$.
\end{enumerate}
A proof of (1) can be found in~\cite[Step~2 of the proof of Theorem~3.7]{Gabai:2021-1}, and (2), (3) are readily verified from the definition.

\subsubsection{Dax homomorphism of $\pi_3(M)$}
\label{subsubsection:dax-homomorphism-pi_3}

As a special case of the above, a homomorphism $d\colon \pi_3(M) \to \Z[\pi_1(M)\setminus 1]$ is defined.
Fix a smooth embedding $\zeta\colon I\hookrightarrow M$, with $\zeta(0)=x_0$ and $\zeta(1)=x_1$.
Note that $\pi_3(M)$ is in 1--1 correspondence with the set of homotopy (rel $\partial I^3$) classes of maps $H\colon I^3\to M$ whose value on the boundary is given by $H(s,t,u)=\zeta(s)$ for $(s,t,u)\in I\times\partial I^2$ and $H(s,t,u)=x_s$ for $(s,t,u)\in \partial I \times I^2$.
Such $H$ is called a \emph{kernel map} in~\cite{Gabai:2021-1}.
Due to~\cite[Theorem~3.15]{Kosanovic-Teichner:2024-2}, $\dax(H)$ lies in the subgroup
\[
  \Z[\pi_1(M)\setminus 1]^\sigma := \{ r \in \Z[\pi_1(M)\setminus 1] \;|\; \widebar r = r\}
\]
of elements fixed by the standard involution $\overline{\sum_g \lambda_g g} = \sum_g \lambda_g g^{-1}$.

\begin{definition}
  \label{definition:dax-homomorphism-pi_3}
  Define $d_\zeta\colon \pi_3(M) \to \Z[\pi_1(M)\setminus 1]^\sigma$ by $d_{\zeta}(\alpha) := \dax(H)$, where $H$ is a kernel map $I^3 \to M$ representing $\alpha\in \pi_3(M)$.
\end{definition}

Since the value of $\dax(H)$ is invariant under homotopy rel $\partial I^3$, the map $d_\zeta$ is well-defined on~$\pi_3(M)$.
Since the group operation on $\pi_3(M)$ corresponds to the concatenation of maps $H$, the additivity of $\dax(H)$ implies that $d_\zeta$ is a group homomorphism.

For brevity, we denote $d_\zeta$ by $d$ when the choice of $\zeta$ is clearly understood.

\subsubsection{Dax invariant for smooth disks}
\label{subsubsection:dax-for-smooth-disks}

Let $D_0$ and $D_1$ be smooth disks in a smooth 4-manifold $M$ which are homotopic rel~$\partial$.
Fix a quotient map $\phi\colon I^2\to D^2$ which restricts to a diffeomorphism of $\inte(I^2)$ onto $\inte(D^2)$ and sends $0\times I$, $1\times I$, $I\times 0$ and $I\times 1$ to $(-1,0)$, $(1,0)$ and lower and upper semicircles of $\partial D^2$ respectively, as illustrated below.
\begin{equation}
  \vcenter{\hbox{
  \begin{tikzpicture}[
    x=1.6cm, y=1.6cm,
    line width=.8pt,
    edge/.style={-stealth,shorten <=2.5*\r,shorten >=2.5*\r},
  ]
    \small
    \def\r{1pt}
    \fill (0,0) circle(\r) (1,0) circle(\r) (1,1) circle(\r) (0,1) circle(\r);
    \draw [edge] (0,0) -- (1,0) node[pos=.5,below]{$\zeta_-$};
    \draw [edge,-] (1,0) -- (1,1) node[pos=.5,right]{$x_1$};
    \draw [edge] (0,1) -- (1,1) node[pos=.5,above]{$\zeta_+$};
    \draw [edge,-] (0,0) -- (0,1) node[pos=.5,left]{$x_0$};
    \draw [-stealth] (1.5,.5) -- ++(.5,0) node[midway,above]{$\phi$};
    \begin{scope}[shift={(2.5,0)}]
      \fill (0,.5) circle(\r) node[left]{$x_0$} (1,.5) circle(\r) node[right]{$x_1$};
      \draw (0,.5) [edge] arc(180:0:.5) node[midway,above]{$\zeta_+$};
      \draw (0,.5) [edge] arc(180:360:.5) node[midway,below]{$\zeta_-$};
    \end{scope}
  \end{tikzpicture}
  }}
  \label{equation:disk-reparametrization-phi}
\end{equation}
View $D_i$ as a map $D_i\colon I^2\to M$ ($i=0,1$) by composing $D_i\colon D^2\hookrightarrow M$ with~$\phi$.
For a homotopy $H\colon I^3\to M$ from $D_0$ to $D_1$ rel~$\partial$, $\dax(H)$ lies in $\Z[\pi_1(M)\setminus 1]^\sigma$ by~\cite[Corollary 4.12]{Kosanovic-Teichner:2024-2}.
Let $\zeta_+\colon I\hookrightarrow M$ be the upper semicircle of $\partial D_0=\partial D_1$.

\begin{definition}
  \label{definition:smooth-dax-disks}
  The \emph{Dax invariant} for $(D_0,D_1)$ is defined by
  \[
    \Dax(D_0,D_1) = \dax(H) \in \Z[\pi_1 M \setminus 1]^\sigma / d_{\zeta_+}(\pi_3(M)).
  \]
\end{definition}

It is known that the modulo $d_{\zeta_+}(\pi_3(M))$ value of $\Dax(D_0,D_1)$ is well-defined, independent of the choice of a homotopy~$H$~\cite{Gabai:2021-1,Kosanovic-Teichner:2024-2}.

If $D_0$, $D_1$ and $D_2$ are smooth disks mutually homotopic rel $\partial$ in~$M$,
then $\Dax(D_0,D_2) = \Dax(D_0,D_1) + \Dax(D_1,D_2)$.
It readily follows from the additivity of $\dax(H)$.
Also, $\dax(H^{-1}) = -\dax(H)$ implies $\Dax(D_1,D_0) = -\Dax(D_0,D_1)$.

When we need to distinguish the Dax invariant of smooth disks from the topological version which we define below, we write $\Dax^\sm(D_0,D_1)$ for $\Dax(D_0,D_1)$.

\subsection{Dax invariant for topological disks}
\label{subsection:dax-for-top-disks}

The goal of this subsection is to generalize the smooth Dax invariant to topological disks in a smooth 4-manifold.
Our method for the topological case is a smooth approximation approach:
the Stable Smoothing Theorem~\ref{theorem:stable-smoothing} provides smoothly embedded disks that approximate topological disks, and we prove that the value of the smooth Dax invariant of these disks is well-defined, independent of the choice of smoothing (Theorem~\ref{theorem:dax-top-well-defined}), and invariant under topological isotopy (Theorem~\ref{theorem:dax-top-isotopy}).
In the next subsection, we develop the topological Dax invariant in a topological 4-manifold.

\begin{remark}
  \label{remark:dax-why-smoothing-approach}
  As an alternative attempt to define the topological Dax invariant, one might consider developing a topological version of Dax's original approach in Section~\ref{subsection:preliminary-smooth-dax}, in order to directly repeat Definition~\ref{definition:smooth-dax-for-homotopy}.
  The key is Dax's result that one can perturb a given homotopy between \emph{smooth} disks so that the associated honest track $I^3\to M\times I^2$ is a generic immersion~\cite[Chapters III and~VI]{Dax:1972-1}\@.
  We do not know how to do this in the topological case:
  Dax's perturbation relies on \emph{Thom's jet transversality}, for which no topological analog is known to the authors.
  (One could apply topological transversality to the honest track of a homotopy, but the resulting generic immersion would no longer be the honest track of a homotopy;
  to apply Definition~\ref{definition:smooth-dax-for-homotopy}, a homotopy with a generic honest track is required.)
  Our approach avoids this difficulty.
\end{remark}

The first step of our approach is to show the topological isotopy invariance of the smooth Dax invariant.
View a topological isotopy $\{h_t\colon D^2 \hookrightarrow M\}_{0\le t\le 1}$ between two disks as a homotopy $H\colon I^2 \times I \to M$, $H(s,t,u) = h_u(\phi(s,t))$, where $\phi\colon I^2\to D^2$ is the reparametrization in~\eqref{equation:disk-reparametrization-phi}.

\begin{lemma}[Topological invariance of the smooth Dax invariant]
  \label{lemma:smooth-dax-top-invariance}
  Let $M$ be a smooth 4-manifold and $H\colon I^3 \to M$ be a map satisfying the boundary condition~\ref{item:dax-boundary-condition}\@.
  If the honest track $\widebar H\colon I^3\to M\times I^2$ is a topological embedding, then $\dax(H)=0$ in $\Z[\pi_1(M)\setminus 1]$.
  Consequently, if two smooth disks $D_0$ and $D_1$ in $M$ are topologically isotopic, then $\Dax^\sm(D_0,D_1) = 0$.
\end{lemma}

This is essentially another application of the principle described in Lemma~\ref{lemma:pi_1-null-approx}: an approximation close to a topological embedding has null-homotopic double point loops.

\begin{proof}[Proof of Lemma~\ref{lemma:smooth-dax-top-invariance}]
  Perturb $H$ by a homotopy with sufficiently small diameter to a smooth map $F\colon I^3 \to M$ so that $\widebar{F}\colon I^3 \to M\times I^2$ is a generic track of~$H$.
  Since $F$ is close to $H$, $\widebar F$ is close to the topological embedding $\widebar H$.
  By Lemma~\ref{lemma:pi_1-null-approx}, $g(p)=1$ for every double point $p$ of $\widebar{F}$.
  It follows that $\dax(H)=0$.
\end{proof}

Now, let $D_0$ and $D_1$ be topological disks in a smooth 4-manifold~$M$.
Choose $k\ge 0$ and smooth disks $D_i'$ in $M\# k(S^2\times S^2)$ admitting a topological isotopy $h_i\colon I^2\times I \to M \# k(S^2\times S^2)$ from $D_i$ to $D_i'$, $i=0,1$.
Here, the connected sum with $k(S^2\times S^2)$ is performed at a point $q\in\inte(M) \setminus(D_0\cup D_1)$.
The Stable Smoothing Theorem~\ref{theorem:stable-smoothing} ensures that this is always possible.
For brevity, we write $h_i \colon D_i \approx D_i'$.
Suppose that $D_0$ and $D_1$ are homotopic rel~$\partial$ in~$M$.
We can always perturb a homotopy to avoid~$q$, by topological transversality (see 9.5 and the first two paragraphs of 9.6 of~\cite{Freedman-Quinn:1990-1}).
Let $H\colon I^3 \to M\setminus q$ be a homotopy $D_0\simeq D_1$ rel~$\partial$.
Viewing $M\setminus q$ as a subspace of $M\#k(S^2\times S^2)$, the concatenation $H' = h_0^{-1} \cdot H \cdot h_1$ is a homotopy from $D_0'$ to $D_1'$ rel $\partial$ in $M\# k(S^2\times S^2)$.
We write:
\[
  H'\colon D_0' \smash{\stackrel{h_0^{-1}}{\approx}} D_0 \stackrel{\smash{H}}{\simeq} D_1 \stackrel{h_1}{\approx} D_1'.
\]
Since $H'$ is between \emph{smooth} disks, $\dax(H')$ is defined in $M\# k(S^2\times S^2)$ by Definition~\ref{definition:smooth-dax-for-homotopy}, as an element in $\Z[\pi_1(M \# k(S^2\times S^2)) \setminus 1]^\sigma=\Z[\pi_1(M) \setminus 1]^\sigma$.

\begin{definition}[Dax invariants for topological disks in a smooth 4-manifold]
  \label{definition:dax-top}
  Define
  \[
    \Dax^\top(D_0, D_1) := \Dax(H') \in \Z[\pi_1(M) \setminus 1]^\sigma / d_{\zeta_+}(\pi_3(M))
  \]
  where $\zeta_+$ is the upper semicircle of $\partial D_0 = \partial D_1$ and $d_{\zeta_+}\colon \pi_3(M) \to \Z[\pi_1(M) \setminus 1]^\sigma$ is the Dax homomorphism in Definition~\ref{definition:dax-homomorphism-pi_3}.
\end{definition}

Note that we use $d_{\zeta_+}(\pi_3(M))$ in Definition~\ref{definition:dax-top}, instead of $d_{\zeta_+}(\pi_3(M\# k(S^2\times S^2)))$, even though $D_i'$ and $H'$ are not in~$M$.
Despite this, the following hold.

\begin{theorem}
  \label{theorem:dax-top-well-defined}
  The value of $\Dax^\top(D_0, D_1)$ is well-defined, independent of the choice of smoothings $D_0'$ and $D_1'$ and the choice of a homotopy $H$ from $D_0$ to~$D_1$.
\end{theorem}

\begin{proof}
  We first verify the independence of the choice of smoothings~$D_i'$ and isotopies $h_i\colon D_i \approx D_i'$ ($i=0,1$).
  Fix a homotopy $H\colon D_0 \simeq D_1$ rel $\partial$ in $M\setminus q$.
  Write $H':=h_0^{-1}\cdot H\cdot h_1$ as before.
  Let $D_i''$ be another smooth disk in $M\# k' (S^2\times S^2)$ admitting a topological isotopy $g_i\colon D_i \approx D_i''$, $i=0,1$.
  Let
  \[
    H'' = g_0^{-1}\cdot H \cdot g_1 \colon D_0''\approx D_0 \simeq D_1 \approx D_1''
  \]
  be the concatenation.
  We will show that $\dax(H')=\dax(H'')$ in $\Z[\pi_1(M)\setminus 1]^\sigma$.
  Since $\dax(-)$ is unchanged by stabilization of the ambient manifold, we may assume that $k=k'$.
  The homotopy $h_i^{\vphantom{1}} \cdot h_i^{-1} \colon D_i \approx D_i' \approx D_i$ is homotopic (rel $\partial I^3$) to the constant homotopy of~$D_i$.
  So $H''$ is homotopic (rel $\partial I^3$) to $g_0^{-1} (h_0 \cdot h_0^{-1}) H (h_1 h_1^{-1}) g_1 = (g_0^{-1} h_0) H' (h_1^{-1} g_1)\colon$
  \[
    D_0'' \stackrel{g_0^{-1}}{\approx} D_0 \stackrel{h_0}{\approx} D_0' \stackrel{h_0^{-1}}{\approx} D_0 \stackrel{H}{\simeq} D_1 \stackrel{h_1}{\approx} D_1' \stackrel{h_1^{-1}}{\approx} D_1 \stackrel{g_1}{\approx} D_1'',
  \]
  Each of $H'$, $H''$, $g_0^{-1}h_0$ and $h_1^{-1}g_1$ is a homotopy between smooth disks.
  Thus, by the homotopy invariance and additivity of $\dax(-)$ (see Section~\ref{subsection:preliminary-smooth-dax}), we have
  \[
    \dax(H'') = \dax( (g_0^{-1} h_0) H' (h_1^{-1} g_1) ) = \dax(g_0^{-1} h_0) + \dax(H') + \dax(h_1^{-1} g_1).
  \]
  Since $g_0^{-1}h_0$ and $h_1^{-1}g_1$ are topological isotopies, we have $\dax(g_0^{-1}h_0) = \dax(h_1^{-1}g_1) = 0$ by Lemma~\ref{lemma:smooth-dax-top-invariance}.
  It follows that $\dax(H') = \dax(H'')$, as promised.

  Now, fix topological isotopies $h_i\colon D_i \approx D_i'$, and let $H$, $K \colon D_0 \simeq D_1$ rel $\partial I^3$ be homotopies in~$M\setminus q$.
  We will show that $\dax(h_0^{-1}H h_1) = \dax(h_0^{-1}K h_1)$ modulo $d(\pi_3(M))$.
  By the additivity, inversion and homotopy invariance, we have
  \[
    \dax(h_0^{-1}H h_1) - \dax(h_0^{-1}K h_1) = \dax(h_0^{-1}H h_1 \cdot h_1^{-1}K^{-1} h_0) = \dax(h_0^{-1}HK^{-1}h_0).
  \]
  Thus, it suffices to show that $\dax(h_0^{-1}HK^{-1}h_0)$ lies in $d(\pi_3(M))$.
  
  To prove this, we will use a reparametrization, following~\cite[Section~4.1, Equation~(4.5)]{Kosanovic-Teichner:2024-2}.
  Let $\psi\colon I^2\to I^2$ be a quotient map which collapses $0\times I$ to a point, identifies $(t,0)$ with $(t,1)$ ($0\le t\le 1$) and restricts to an orientation preserving diffeomorphism of $
  \inte(I^2)$ onto its image, as illustrated below:
  
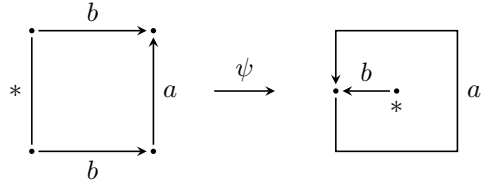
\[
    \begin{tikzpicture}[
      x=1.6cm, y=1.6cm,
      line width=.6pt,
      edge/.style={-stealth,shorten <=2.5*\r,shorten >=2.5*\r},
    ]
      \small
      \def\r{1pt}
      \fill (0,0) circle(\r) (1,0) circle(\r) (1,1) circle(\r) (0,1) circle(\r);
      \draw [edge] (0,0) -- (1,0) node[midway,below]{$b$};
      \draw [edge] (1,0) -- (1,1) node[midway,right]{$a$};
      \draw [edge] (0,1) -- (1,1) node[midway,above]{$b$};
      \draw [edge,-] (0,0) -- (0,1) node[midway,left]{$*$};
      \draw [-stealth] (1.5,.5) -- ++(.5,0) node[midway,above]{$\psi$};
      \begin{scope}[shift={(2.5,0)}]
        \draw [edge] (0,.5) -- (0,0) -- (1,0)
          -- (1,1) node[pos=.5,right]{$a$} -- (0,1) -- (0,.5);
        \fill (.5,.5) circle(\r) node[below]{$*$} (0,.5) circle(\r);
        \draw [edge] (.5,.5) -- (0,.5) node[midway,above]{$b$};
      \end{scope} 
    \end{tikzpicture}
  \]
  For a self homotopy $G\colon I^3\to M$ of a disk rel $\partial$, we can reparametrize $G$ by $\psi$, to obtain a map $G^\psi \colon I^3 \to M$ given by $G^\psi(s,\psi(t,u)) = G(s,t,u)$.
  Since the first coordinate $s$ is unchanged, double points counted to define $\dax(G)$ are in 1--1 correspondence with those used to define $\dax(G^\psi)$.
  The corresponding $\pi_1$ elements are identical since $\psi|_{\inte(I^2)}$ is homotopic to the identity.
  It follows that $\dax(G)=\dax(G^\psi)$ in $\Z[\pi_1(M)\setminus 1]^\sigma$.

  Define the adjoint $G^{ad}\colon I^2 \to \Map(I, M)$ by $G^{ad}(t,u)(s) = G(s,t,u)$.
  Then we have $(G^\psi)^{ad} \circ \psi = G^{ad}$.
  For the case of $G:=h_0^{-1}HK^{-1}h_0$, $G^{ad}$ and $(G^\psi)^{ad}$ are as follows:
  \[
    G^{ad} = \vcenter{\hbox{
    \begin{tikzpicture}[
      x=.8cm, y=.8cm,
      line width=.6pt,
      edge/.style={-stealth,shorten <=2.5*\r,shorten >=2.5*\r},
    ]
      \small
      \def\r{1pt}
      \foreach \i in {0,1,2,3,4} { \fill (0,\i) circle(\r) (4,\i) circle(\r); }
      \foreach \i in {0,1,2,3} { \draw [edge,-] (0,\i) -- ++(0,1) ; \draw [edge,-] (4,\i) -- ++(0,1); }
      \foreach \i/\l in {0/D_0', 1/D_0, 2/D_1, 3/D_0, 4/D_0'} {
        \tiny \draw [edge] (0,\i) -- (4,\i) node[pos=.33,fill=white]{$\l$};
      }
      \draw (2.5,0.5) node{$h_0^{-1}$} (2.5,1.5) node{$H$} (2.5,2.5) node{$K^{-1}$} (2.5,3.5) node{$h_0$};
      \tiny \draw (0,2) node[left]{$\zeta_-$} (4,2) node[right]{$\zeta_+$};
    \end{tikzpicture}
    }}
    , \quad
    (G^\psi)^{ad} = \vcenter{\hbox{
    \begin{tikzpicture}[
      x=.8cm, y=.8cm,
      line width=.6pt,
      edge/.style={-stealth,shorten <=2.5*\r,shorten >=2.5*\r},
    ]
      \small
      \def\r{1pt}
      \fill (0,0) circle(\r) (4,0) circle(\r) (0,2) circle(\r) (4,2) circle(\r)
        (0,4) circle(\r) (4,4) circle(\r) (2,2) circle(\r);
      \tiny
      \draw [edge,-] (0,0)--(4,0) node[midway,below]{$\zeta_+$};
      \draw [edge,-] (0,4)--(4,4) node[midway,above]{$\zeta_+$};
      \draw [edge,-] (0,0)--(0,2) node[midway,left]{$\zeta_+$};
      \draw [edge,-] (0,2)--(0,4) node[midway,left]{$\zeta_+$};
      \draw [edge,-] (4,0)--(4,2) node[midway,right]{$\zeta_+$};
      \draw [edge,-] (4,2)--(4,4) node[midway,right]{$\zeta_+$};
      \draw [edge] (2,2)--(4,2) node[midway,fill=white]{$D_1$};
      \draw [edge] (2,2)--(0,4) node[midway,fill=white]{$D_0$};
      \draw [edge] (2,2)--(0,2) node[midway,fill=white]{$D_0'$};
      \draw [edge] (2,2)--(0,0) node[midway,fill=white]{$D_0$};
      \draw (2,2) node[shift={(1ex,-2ex)}]{$\zeta_-$};
      \small
      \draw (.5,2.5) node{$h_0^{-1}$} (.5,1.5) node{$h_0$} (3,3) node{$K^{-1}$} (3,1) node{$H$};
    \end{tikzpicture}
    }}
  \] 
  In the diagram of $(G^\psi)^{ad}$, the two triangles representing $h_0^{-1}$ and $h_0$ are eliminated by homotopy rel~$\partial$.
  It follows that $G^\psi$ is homotopic (rel~$\partial I^3$) to another homotopy $G'$ which is the concatenation of reparametrizations of $H$ and~$K^{-1}$.
  Since $H$ and $K$ are homotopies in $M$, so is $G'$, i.e.\ $G'\colon I^3 \to M$.
  Since $(G^\psi)^{ad}$ has the constant value $\zeta_+$ on $\partial I^2$, so does $(G')^{ad}$.
  Therefore $G'$ represent an element of $\pi_3(M)$ in the sense of Section~\ref{subsubsection:dax-homomorphism-pi_3}.
  It follows that
  \[
    \dax(h_0^{-1}H K^{-1} h_0) = \dax(G) = \dax(G^\psi) = \dax(G') \in d(\pi_3(M)).
    \qedhere
  \]
\end{proof}

\begin{remark}[Dax for disks with geometric dual]
  Suppose that two topological disks $D_0$ and $D_1$ are homotopic rel~$\partial$ in a smooth 4-manifold $M$ and suppose that each $D_i$ has a smooth geometric dual in~$M$.
  Especially, it is the case if a meridian of the common boundary $K=\partial D_0=\partial D_1$ is trivial in $\pi_1(\partial M\setminus K)$, by Lemma~\ref{lemma:dual-in-3-manifold}.
  Then, we do not need to stabilize $M$ to define $\Dax^\top(D_0,D_1)$, since $D_i$ has a smoothing $D_i'$ in $M$ by the Light Bulb Smoothing Theorem~\ref{theorem:light-bulb-smoothing}\@.
\end{remark}

For later use, we state basic properties of $\Dax^\top$ below.

\begin{theorem}
  \label{theorem:dax-top-properties}
  Let $D_0$, $D_1$ and $D_2$ are topological disks in a smooth 4-manifold $M$ which are mutually homotopic rel~$\partial$.
  Then the following hold.
  \begin{enumerate}
    \item\label{item:dax-top-additivity} $\Dax^\top(D_0,D_2) = \Dax^\top(D_0,D_1) + \Dax^\top(D_1,D_2)$.
    \item\label{item:dax-top-inverse} $\Dax^\top(D_0,D_1) = -\Dax^\top(D_1,D_0)$
    \item\label{item:dax-top-smooth} If the disks $D_0$ and $D_1$ are smooth, $\Dax^\top(D_0,D_1) = \Dax^\sm(D_0,D_1)$.
  \end{enumerate}
\end{theorem}

\begin{proof}
  \ref{item:dax-top-additivity}~and \ref{item:dax-top-inverse} follow readily from the additivity and inversion formulas for $\Dax^\sm$ in Section~\ref{subsection:preliminary-smooth-dax}.
  \ref{item:dax-top-smooth}~is obvious from Definition~\ref{definition:dax-top}.
\end{proof}

\begin{theorem}
  \label{theorem:dax-top-isotopy}
  If two topological disks $D_0$ and $D_1$ in a smooth 4-manifold $M$ are topologically isotopic, then $\Dax^\top(D_0,D_1)=0$.
  More generally, $\Dax^\top(D_0,D_1)$ is invariant under topological isotopy of $D_0$ and~$D_1$.
\end{theorem}

\begin{proof}
  Let $h_i$ be a topological isotopy from $D_i$ to a smoothing $D_i'$ in $M\#k(S^2\times S^2)$ as in Definition~\ref{definition:dax-top}.
  If $D_0$ and $D_1$ are topologically isotopic in $M\setminus q$, then for a topological isotopy $H\colon I^3\to M\setminus q$, $H'=h_0^{-1}\cdot H \cdot h_1$ is a topological isotopy from $D_0'$ to $D_1'$ in $M\#k(S^2\times S^2)$.
  It follows that $\Dax^{\smash{\top}}(D_0,D_1) = \Dax(H')=0$ by Lemma~\ref{lemma:smooth-dax-top-invariance}.
  From this and Lemma~\ref{lemma:isotopy-avoiding-points} below, the first sentence of Theorem~\ref{theorem:dax-top-isotopy} readily follows.
  The second sentence follows from Theorem~\ref{theorem:dax-top-properties}~\ref{item:dax-top-additivity} and the first sentence.
\end{proof}

\begin{lemma}
  \label{lemma:isotopy-avoiding-points}
  Suppose that $f_0$, $f_1\colon R\hookrightarrow M$ are embeddings of a surface $R$ with nonempty boundary into a topological 4-manifolds~$M$, and $q \in \inte(M)\setminus (f_0(R)\cup f_1(R))$.
  If $f_0$ and $f_1$ are isotopic (rel~$\partial$) in $M$, then $f_0$ and $f_1$ are isotopic in $M\setminus q$.
\end{lemma}

\begin{proof}
  Fix collars $\partial M\times[0,1]$ and $\partial R\times[0,1]$ of $\partial M\times 0=\partial M \subset M$ and $\partial R\times 0 = \partial R \subset R$.
  By isotopy, we may assume that $f_0(x,s)=f_1(x,s)=(f_0(x,0),s)\in M\times[0,1]$ for $(x,s)\in \partial R\times[0,1]$.
  By identifying $M\setminus (\partial M\times [0,1))$ with $M$, there is an isotopy $\{f_t\colon R\hookrightarrow M\}$ from $f_0$ to $f_1$ such that $f_t=f_0$ on $\partial R\times[0,1]$ and $f_t(R\setminus (\partial R\times[0,1])) \subset M\setminus (\partial M\times[0,1])$.
  Choose $q'$ in $(\partial M\times \frac12) \setminus f_0(R)$.
  Since $f_0$ and $f_1$ are locally flat embeddings, $M\setminus (f_0(R)\cup f_1(R))$ is connected.
  An isotopy sending $q'$ to $q$ along an arc in $M\setminus (f_0(R)\cup f_1(R))$ induces, by isotopy covering, a homeomorphism $h\colon M\to M$ such that $h(q')=q$ and $h$ fixes $f_0(R) \cup f_1(R)$ pointwise.
  Define an isotopy $\{g_t\colon R \hookrightarrow M\}$ by $g_t = h \circ f_t$.
  We have $g_0=f_0$, $g_1=f_1$.
  Also, $q'\notin f_t(R)$ implies that $q=h(q') \notin h(f_t(R)) = g_t(R)$.
  So, $\{g_t\}$ is an isotopy from $f_0$ to $f_1$ in $M\setminus q$.
\end{proof}

Note that the same argument shows a homotopy analog of Lemma~\ref{lemma:isotopy-avoiding-points}.

When $M$ has a smooth structure, Theorem~\ref{theorem:dax-top} stated in the introduction is a combination of Theorems~\ref{theorem:dax-top-well-defined} and~\ref{theorem:dax-top-isotopy}.

\subsection{Dax invariant for disks in a topological 4-manifold}
\label{subsection:dax-in-top-4-manifold}

Let $M$ be a topological 4-manifold.
Fix an arbitrary almost smooth structure on~$M$, i.e.,\ fix an interior point $p_0$ in $M$ and fix a smooth structure on $M\setminus p_0$.
An almost smooth structure exists by~\cite[Theorem~8.2]{Freedman-Quinn:1990-1}.
Suppose $D_0$ and $D_1$ are disks in~$M$.
Assume that each $D_i$ does not contain~$p_0$, by isotopy of $D_i$ if necessary.
Recall that the Dax invariant of the topological disks $D_0$ and $D_1$ in the smooth 4-manifold $M\setminus p_0$ has been defined in Definition~\ref{definition:dax-top}.
To specify the ambient manifold explicitly, we denote it by
\begin{equation}
  \begin{aligned}
    \Dax^\top_{M\setminus p_0}(D_0,D_1) & \in \Z[\pi_1(M-p_0)\setminus 1]^\sigma / d(\pi_3(M \setminus p_0))
    \\
    & = \Z[\pi_1(M)\setminus 1]^\sigma / d(\pi_3(M \setminus p_0)).
  \end{aligned}
  \label{equation:dax-in-M-p_0}
\end{equation}

We will use the following lemma to express $d(\pi_3(M \setminus p_0))$ in~\eqref{equation:dax-in-M-p_0} in terms of~$M$.

\begin{lemma}
  \label{lemma:dax-homomorphism-on-embedded-3-sphere}\phantomsection
  \begin{enumerate}
    \item Let $M$ be a smooth 4-manifold and $\zeta\colon I\hookrightarrow M$ be a smooth embedding.
    If a class $\alpha\in\pi_3(M)$ is represented by a topologically embedded (locally flat) 3-sphere in $M$, then $\alpha$ lies in the kernel of $d_{\zeta}\colon \pi_3(M)\to \Z[\pi_1(M)\setminus 1]^\sigma$.
    \item Let $M$ be a topological 4-manifold equipped with an almost smooth structure with singular point~$p_0$.
    Then the Dax homomorphism $d \colon \pi_3(M\setminus p_0)\to \Z[\pi_1(M)\setminus 1]^\sigma$ induces a homomorphism $d\colon \pi_3(M)\to \Z[\pi_1(M)\setminus 1]^\sigma$ such that $d(\pi_3(M\setminus p_0)) = d(\pi_3(M))$.
  \end{enumerate}
\end{lemma}

\begin{proof}
  (1) Write $S^3$ as the union of hemispheres $I^3_- \cup_\partial I^3_+$, use $*=(\frac12,\frac12,\frac12)\in I^3_- = I^3$ as the basepoint for $S^3$ and use $\zeta(\frac12)\in M$ as the basepoint for~$\pi_3(M)$.
  Let $f\colon S^3\to M$ be a topological embedding which represents $\alpha\in \pi_3(M)$.
  By isotopy, we may assume that $\zeta(I) \subset f(S^3)$ and $f(*) = \zeta(\frac12)$.
  Then, we can homotope $f$ to a map $g\colon S^3=I^3_- \cup_\partial I^3_+ \to M$ such that $g$ is injective on $\inte(I^3_+)$ and $g$ on ${I^3_-}$ is given by $g(s,t,u)=\zeta(s)$.
  This means that the restriction $H\colon I^3 = I^3_+ \to M$ of $g$ on $I^3_+$ is a kernel map with respect to~$\zeta$.
  Since $H$ represents~$\alpha$ and is injective on the interior, $d_\zeta(\alpha)=\dax(H)=0$ by Lemma~\ref{lemma:smooth-dax-top-invariance}.

  (2) The inclusion induces a surjection $\pi_3(M\setminus p_0)\to \pi_3(M)$, whose kernel is the $\Z[\pi_1(M)]$-submodule generated by the boundary 3-sphere of a 4-ball neighborhood of~$p_0$, by a general position argument.
  The conclusion follows from this and~(1).
\end{proof}

We use Lemma~\ref{lemma:dax-homomorphism-on-embedded-3-sphere}~(2) in the following definition.

\begin{definition}[Dax invariant for disks in a topological 4-manfold]
  \label{definition:dax-in-top-4-manifold}
  Define the \emph{Dax invariant} in the topological 4-manifold $M$ by
  \[
    \Dax^\top_M(D_0,D_1)=\Dax^\top_{M\setminus p_0}(D_0,D_1) \in \Z[\pi_1(M)\setminus 1]^\sigma / d(\pi_3(M)).
  \]
\end{definition}

Theorem~\ref{theorem:dax-in-top-4-manifold-properties} stated below readily follows from Theorem~\ref{theorem:dax-top-properties}.

\begin{theorem}
  \label{theorem:dax-in-top-4-manifold-properties}
  Let $D_0$, $D_1$ and $D_2$ are disks in a topological 4-manifold $M$, which are mutually homotopic rel~$\partial$.
  Then the following hold.
  \begin{enumerate}
    \item\label{item:dax-in-top-4-manifold-additivity} $\Dax^\top_M(D_0,D_2) = \Dax^\top_M(D_0,D_1) + \Dax^\top_M(D_1,D_2)$.
    \item\label{item:dax-in-top-4-manifold-inverse} $\Dax^\top_M(D_0,D_1) = -\Dax^\top_M(D_1,D_0)$
  \end{enumerate}
\end{theorem}

\begin{theorem}
  \label{theorem:dax-in-top-4-manifold-isotopy}
  The value of $\Dax^\top_M(D_0,D_1)$ is invariant under isotopy of $D_0$ and $D_1$ in the topological 4-manifold~$M$.
  If $D_0$ and $D_1$ are isotopic in $M$, $\Dax^\top_M(D_0,D_1)$ vanishes.
\end{theorem}

\begin{proof}
  By Theorem~\ref{theorem:dax-top-isotopy}, $\Dax^\top_M(D_0,D_1):=\Dax^\top_{M\setminus p_0}(D_0,D_1)$ is invariant under topological isotopy in $M\setminus p_0$.
  The conclusion readily follows from this and Lemma~\ref{lemma:isotopy-avoiding-points}.
\end{proof}

\begin{remark}
  Recall that we isotope $D_i$ to avoid the singular point $p_0$ in Definition~\ref{definition:dax-in-top-4-manifold}.
  Theorem~\ref{theorem:dax-in-top-4-manifold-isotopy} ensures that the value of $\Dax^\top_M(D_0,D_1)$ is independent of this isotopy.
  So $\Dax^\top_M(D_0,D_1)$ is well-defined even when $D_i$ contains~$p_0$.
\end{remark}

\section{Proofs of Corollaries~\ref{corollary:top-equal-smooth-disk-isotopy}, \ref{corollary:disk-top-isotopy-classification} and~\ref{corollary:knotted-3-disk}}
\label{section:proof-top=smooth-for-disk}

First we prove Corollary~\ref{corollary:disk-top-isotopy-classification}\@.
To recall the statement, let $M$ be a topological 4-manifold and $D_0$ be a topological disk in $M$ such that a meridian of $D_0$ is null-homotopic in $\partial M \setminus \partial D_0$.
Let $\cD^\top_M(D_0)$ be the set of isotopy classes of disks in $M$ homotopic to $D_0$ rel~$\partial$.
Corollary~\ref{corollary:disk-top-isotopy-classification} asserts that
\[
  \Dax^\top_M(D_0,-)\colon
  \cD^\top_M(D_0)
  \to
  \Z[\pi_1(M)\setminus 1]^\sigma / d(\pi_3(M))
\]
defined in Definition~\ref{definition:dax-in-top-4-manifold} is a bijection.

As a special case, suppose that $M$ has a smooth structure.
Proof of the general case and the proof of Corollary~\ref{corollary:top-equal-smooth-disk-isotopy} rely on this special case.

\begin{proof}[Proof of Corollary~\ref{corollary:disk-top-isotopy-classification}, when $M$ has a smooth structure]
  We may assume that the disk $D_0$ is smooth in~$M$, by the Light Bulb Smoothing Theorem~\ref{theorem:light-bulb-smoothing}\@.
  Let $\cD^\sm_M(D_0)$ be the set of smooth isotopy classes of smooth disks in $M$ homotopic to $D_0$ rel~$\partial$.
  Consider the following diagram, where $\Phi\colon \cD^\sm_M(D_0) \to \cD^\top_M(D_0)$ is the map induced by the identity, and $\Dax^\top_M(D_0,-)$ is the invariant in Definition~\ref{definition:dax-top}.
  \[
    \begin{tikzcd}[row sep=tiny]
      \cD^\sm_M(D_0) \ar[dd,"\Phi"']
        \ar[rd,"{\Dax^\sm_M(D_0,-)}",end anchor={[yshift=-1.2ex]north west},pos=.3]
      \\
      & \Z[\pi_1(M)\setminus 1]^\sigma / d(\pi_3(M))
      \\
      \cD^\top_M(D_0)
        \ar[ru,"{\Dax^\top_M(D_0,-)}"',end anchor={[yshift=1.2ex]south west},pos=.3]
    \end{tikzcd}  
  \]
  The map $\Dax^\top_M(D_0,-)$ is well-defined on $\cD^\top_M(D_0)$ by Theorem~\ref{theorem:dax-top-isotopy}.
  The diagram is commutative by Theorem~\ref{theorem:dax-top-properties}~\ref{item:dax-top-smooth}.
  By the Light Bulb Smoothing Theorem~\ref{theorem:light-bulb-smoothing}, $\Phi$~is onto.
  By the classification of $\cD^\sm_M(D_0)$ due to Kosanovi\'c and Teichner~\cite{Kosanovic-Teichner:2024-2}, $\Dax^\sm_M(D_0,-)$ is bijective.
  (Theorems~1.1 and~1.5~(3) of~\cite{Kosanovic-Teichner:2024-2} give the injectivity and surjectivity respectively.)
  It follows that $\Phi$ and $\Dax^\top_M(D_0,-)$ are bijective.
\end{proof}

\begin{proof}[Proof of Corollary~\ref{corollary:disk-top-isotopy-classification}, for the general case]
  Let $M$ be a topological 4-manifold.
  Fix an interior point $p_0$ and a smooth structure on $M\setminus p_0$.
  Consider the invariants $\smash{\Dax^\top_{M\setminus p_0}(D_0,-)}$ in Definition~\ref{definition:dax-top} and $\smash{\Dax^\top_M(D_0,-)}$ in Definition~\ref{definition:dax-in-top-4-manifold}.
  By Theorems~\ref{theorem:dax-top-isotopy} and \ref{theorem:dax-in-top-4-manifold-isotopy}, they are well-defined on $\cD^\top_{M\setminus p_0}(D_0)$ and $\cD^\top_M(D_0)$ respectively, and by Definition~\ref{definition:dax-in-top-4-manifold}, the following diagram is commutative, where $\Psi$ is induced by the identity.
  \[
    \begin{tikzcd}[row sep=tiny]
      \cD^\top_{M\setminus p_0}(D_0) \ar[dd,"\Psi"']
        \ar[rd,"{\Dax^\top_{M\setminus p_0}(D_0,-)}",end anchor={[yshift=-1.2ex]north west},pos=.3]
      \\
      & \Z[\pi_1M\setminus 1]^\sigma / d(\pi_3M)
      \\
      \cD^\top_M(D_0)
        \ar[ru,"{\Dax^\top_M(D_0,-)}"',end anchor={[yshift=1.2ex]south west},pos=.3]
    \end{tikzcd}  
  \]
  By the above special case, ${\Dax^\top_{M\setminus p_0}(D_0,-)}$ is bijective.
  Since we can isotope a disk to avoid~$p_0$, $\Psi$ is surjective.
  It follows that ${\Dax^\top_M(D_0,-)}$ is bijective.
\end{proof}

Now, Corollary~\ref{corollary:top-equal-smooth-disk-isotopy} is readily obtained from the above special case proof of Corollary~\ref{corollary:disk-top-isotopy-classification}\@.
Let $M$ be a smooth 4-manifold and $K$ be a fixed circle in $\partial M$ which has a geometric dual in $\partial M$.
For brevity, denote by $\cD^\sm$ and $\cD^\top$ the sets of smooth and topological isotopy classes of smooth and topological disks in $M$ bounded by~$K$, respectively.
Recall that Corollary~\ref{corollary:top-equal-smooth-disk-isotopy} asserts that the natural map $\cD^\sm \to \cD^\top$ is a bijection.

\begin{proof}[Proof of Corollary~\ref{corollary:top-equal-smooth-disk-isotopy}]
  Choose smooth disk representatives $D_\alpha$ of the homotopy (rel~$\partial$) classes of smooth disks bounded by~$K$.
  Then $\cD^\sm$ is partitioned into the disjoint union of~$\cD^\sm_M(D_\alpha)$.
  A topological disk bounded by~$K$ is smoothable by the Light Bulb Smoothing Theorem~\ref{theorem:light-bulb-smoothing}\@ and thus homotopic to some~$D_\alpha$.
  It follows that $\cD^\top$ is partitioned into the disjoint union of~$\smash{\cD^\top_M(D_\alpha)}$.
  In the special case proof of Corollary~\ref{corollary:disk-top-isotopy-classification}, we showed that $\Phi\colon \cD^\sm_M(D_\alpha) \to \cD^\top_M(D_\alpha)$ is bijective.
  Since it holds for each $D_\alpha$, it follows that $\cD^\sm \to \cD^\top$ is bijective.
\end{proof}

We finish this section with a proof of Corollary~\ref{corollary:knotted-3-disk}.

\begin{proof}[Proof of Corollary~\ref{corollary:knotted-3-disk}]
  Our proof follows Gabai's proof of the smooth version~\cite[Proof of Theorem~5.1]{Gabai:2021-1}, replacing the use of smooth isotopy invariance of Dax and the smooth light bulb theorem with their topological analogs that we proved in this paper.
  Details are as follows.

  Let $D_0=*\times D^2$ in the first factor of $M=(S^2\times D^2)\#(S^1\times D^3)$ and $\Delta_0=*\times D^3$ in the second factor of~$M$.
  Note $D_0$ and $\Delta_0$ are disjoint.
  Gabai constructed self-diffeomorphisms $\psi_i$ on $M=(S^2\times D^2)\#(S^1\times D^3)$ ($i\ge 1$) such that $\psi_i\simeq \id_M$ rel $\partial$ and the disks $D_i := \psi_i(D_0)$ satisfy $\Dax^\sm(D_i,D_j)\ne 0$ whenever $i\ne j$~\cite[Lemma~5.3 and Proof of Theorem~5.1]{Gabai:2021-1}.

  Let $\Delta_i := \psi_i(\Delta_0)$.
  Suppose $\Delta_i$ is topologically isotopic to $\Delta_j$, $i\ne j$.
  Then an ambient topological isotopy on $M$ covering the isotopy sends $D_i$ to a topological disk $D'\subset M$ which is disjoint from~$\Delta_j$.
  Let $N_j$ be the exterior of $\Delta_j \subset M$.
  Note that $N_j=\psi_j(N_0) \cong N_0= S^2\times D^2$ and $D'$, $D_j$ are disks in~$N_j$.
  Since the spheres $D'\cup_\partial -D_i$ and $D_i\cup_\partial -D_j$ are null-homotopic in~$M$ and $\pi_2(N_j)\cong H_2(N_j)$ injects into $H_2(M)$, $D'\cup_\partial -D_j$ is null-homotopic in $N_j$, i.e.,\ $D'$ and $D_j$ are homotopic rel $\partial$ in~$N_j$.
  By Corollary~\ref{corollary:disk-top-isotopy-classification}, it follows that $D'$ is topologically isotopic to $D_j$ in~$N_j$, since $\pi_1(N_j)=0$.
  Therefore, the smooth disks $D_i$ and $D_j$ are topologically isotopic in~$M$ and so $\Dax^\top(D_i,D_j) = \Dax^\sm(D_i,D_j)=0$ by Theorem~\ref{theorem:dax-top-isotopy}.
  This is a contradiction.
\end{proof}

\section{Applications to spheres}
\label{section:application-to-spheres}

In this section, we prove sphere analogs of our results on disks shown in Sections~\ref{section:top-dax-invariant} and~\ref{section:proof-top=smooth-for-disk}.
The arguments in this section are also similar to the disk case, except that we use the Freedman-Quinn invariant~\cite{Freedman-Quinn:1990-1,Stong:1993-1} in place of the Dax invariant.
In recent work of Schneiderman and Teichner~\cite{Schneiderman-Teichner:2022-1} and subsequent papers of Klug and Miller~\cite{Klug-Miller:2021-1,Klug-Miller:2022-1}, 
the Freedman-Quinn invariant for smooth spheres is a key ingredient for the study of 4-dimensional light bulb theorems.
We need a topological version of the invariant, which generalizes the smooth version used in~\cite{Schneiderman-Teichner:2022-1,Klug-Miller:2021-1,Klug-Miller:2022-1}.
Specifically, we use the following.

\begin{theorem}
  \label{theorem:fq-top}
  Let $M$ be a topological 4-manifold.
  Then there is a function
  \[
    \FQ^\top_M \colon \left\{ (R_0, R_1) \, \left| \,
    \begin{tabular}{@{}c@{}}
      $R_0$, $R_1$ are homotopic spheres in $M$ \\
      with a common geometric dual sphere
    \end{tabular} \right.\right\}
    \to \bF_2 T_M / \mu(\pi_3(M))
  \]
  satisfying the following.
  \begin{enumerate}
    \item The value of $\FQ^\top_M(R_0,R_1)$ is invariant under topological concordance of $R_0$ and~$R_1$.
    If $R_0$ and $R_1$ are topologically concordant, $\FQ^\top_M(R_0,R_1)=0$.
    \item For any smooth structure on $M$, $\FQ^\top_M(R_0,R_1)$ is equal to the value of the smooth Freedman-Quinn invariant defined in~\cite{Schneiderman-Teichner:2022-1,Klug-Miller:2021-1,Klug-Miller:2022-1} if $R_0$ and $R_1$ are smooth.
  \end{enumerate}
\end{theorem}

Here, $\bF_2 T_M$ is the vector space over the finite field $\bF_2$ of order 2 generated by the set $T_M$ of order 2 elements in $\pi_1(M)$, and $\mu\colon \pi_3(M)\to \bF_2 T_M$ is a self-intersection invariant, which is analogous to the Dax homomorphism $d$ used in Section~\ref{section:top-dax-invariant}.

The outline of the proof of Theorem~\ref{theorem:fq-top} is similar to the smooth category development detailed in~\cite[Section~4]{Schneiderman-Teichner:2022-1} (see also~\cite{Klug-Miller:2021-1,Klug-Miller:2022-1}), but it requires additional tools to approximate maps by ``generic'' maps in the topological category.
Although the technical details seem to be known to experts, we could not find a proof in the literature.
So we provide a proof in an appendix of this paper.

\begin{remark}
  Theorem~\ref{theorem:fq-top} may be compared with Theorem~\ref{theorem:dax-top} which is for the Dax invariant.
  While Theorem~\ref{theorem:fq-top} seems to be known, Theorem~\ref{theorem:dax-top} is new, to the authors' knowledge.
  The approach of the proof of Theorem~\ref{theorem:fq-top} in the appendix is completely different from the proof of Theorem~\ref{theorem:dax-top} in Section~\ref{section:top-dax-invariant}.
  The former uses high-dimensional smoothing theory of Kirby and Siebenmann~\cite{Kirby-Siebenmann:1977-1} and topological immersion theory of Lees~\cite{Lees:1969-1} (see section~\ref{subsection:immersion-smoothing-generic-maps} in the appendix), while the latter relies on our surface smoothing results in dimension~4 (see Section~\ref{section:top-dax-invariant}).
  In fact, the arguments in Section~\ref{section:top-dax-invariant} can also be used, without essential changes, to provide a quick proof of a Freedman-Quinn version of Theorem~\ref{theorem:dax-top}\@.
  This is sufficient for the applications in this paper (Corollaries~\ref{corollary:top-equal-smooth-sphere-isotopy}, \ref{corollary:sphere-top-isotopy-classification}, \ref{corollary:top-equal-smooth-concordance-isotopy}, \ref{corollary:top-equal-smooth-sphere-isotopy-concordance} and~\ref{corollary:sphere-in-top-4-manifold-isotopy-concordance}).
  On the other hand, Theorem~\ref{theorem:fq-top} is stronger because the resulting invariant does not depend on a choice of an almost smooth structure.
  As aforementioned in Remark~\ref{remark:dax-why-smoothing-approach}, it seems difficult, if not impossible, to apply the approach of the appendix to the Dax invariant.
\end{remark}

In what follows, we discuss applications to spheres which are obtained by combining our smoothing results with the topological Freedman-Quinn invariant.
Especially, Corollaries~\ref{corollary:top-equal-smooth-sphere-isotopy}, \ref{corollary:sphere-top-isotopy-classification} and~\ref{corollary:top-equal-smooth-concordance-isotopy} in the introduction readily follow from Corollaries~\ref{corollary:top-equal-smooth-sphere-isotopy-concordance} and~\ref{corollary:sphere-in-top-4-manifold-isotopy-concordance} which are stated and proven below.

Let $M$ be a smooth 4-manifold, $G$ be a framed smooth sphere in~$M$, and $R_0$ be a smooth sphere in $M$ for which $G$ is a geometric dual.
Let $\cR^\sm_M(G,R_0)$ and $\cR^\top_M(G,R_0)$ be the sets of smooth and topological isotopy classes of spheres in $M$ which are homotopic to $R_0$ and have $G$ as a geometric dual, respectively.
Similarly, define $\cC^\sm_M(G,R_0)$ and $\cC^\top_M(G, R_0)$ by replacing ``isotopy'' with ``concordance.''
By Theorem~\ref{theorem:fq-top}, we have the following commutative diagram.
By the main results of~\cite{Schneiderman-Teichner:2022-1}, the top horizontal arrow and $\FQ^\sm_M(R_0,-)$ are bijections.

\begin{equation}
  \begin{tikzcd}[row sep=tiny]
    \cR^\sm_M(G,R_0) \ar[r] \ar[dd]
    & \cC^\sm_M(G,R_0) \ar[dd]
      \ar[rd,"{\FQ^\sm_M(R_0,-)}",end anchor={[yshift=-.8ex]north west},pos=.3]
    \\
    & & \bF_2 T_M / \mu(\pi_3(M))
    \\
    \cR^\top_M(G,R_0) \ar[r]
    & \cC^\top_M(G,R_0)
      \ar[ru,"{\FQ^\top_M(R_0,-)}"',end anchor={[yshift=.8ex]south west},pos=.3]
    \\
  \end{tikzcd}
  \label{equation:top-smooth-isotopy-concordance-R_0}
\end{equation}
Also, denote by $\cR^\sm_M(G)$ and $\cR^\top_M(G)$ the sets of smooth and topological isotopy classes of smooth and topological spheres in $M$ which have $G$ as a geometric dual, respectively.
Define $\cC^\sm_M(G)$ and $\cC^\top_M(G)$ by replacing isotopy with concordance.
Then, we have the following commutative diagram.
\begin{equation}
  \begin{tikzcd}
    \cR^\sm_M(G) \ar[r] \ar[d] & \cC^\sm_M(G) \ar[d]
    \\
    \cR^\top_M(G) \ar[r] & \cC^\top_M(G)
  \end{tikzcd}
  \label{equation:top-smooth-isotopy-concordance-all}
\end{equation}

\begin{corollary}
  \label{corollary:top-equal-smooth-sphere-isotopy-concordance}
  All maps in~\eqref{equation:top-smooth-isotopy-concordance-R_0} and~\eqref{equation:top-smooth-isotopy-concordance-all} are bijections.
\end{corollary}

In particular, for smooth spheres with a common smooth geometric dual in a smooth 4-manifold, the following are equivalent: smooth isotopy, smooth concordance, topological isotopy and topological concordance.

The proof proceeds similarly to that of Corollaries~\ref{corollary:top-equal-smooth-disk-isotopy} and~\ref{corollary:disk-top-isotopy-classification}\@. The only difference is that we use the topological Freedman-Quinn invariant instead of the Dax invariant.

\begin{proof}[Proof of Corollary~\ref{corollary:top-equal-smooth-sphere-isotopy-concordance}]
  In the diagram~\eqref{equation:top-smooth-isotopy-concordance-R_0}, the vertical arrows are surjective by the Light Bulb Smoothing Theorem~\ref{theorem:light-bulb-smoothing}\@.
  Since the top horizontal arrow and $\FQ^\sm_M(R_0,-)$ are bijective by~\cite{Schneiderman-Teichner:2022-1}, it follows that other maps in~\eqref{equation:top-smooth-isotopy-concordance-R_0} are bijective.

  As in the proof of Corollary~\ref{corollary:top-equal-smooth-disk-isotopy}, the assertion for~\eqref{equation:top-smooth-isotopy-concordance-all} follows from the result for~\eqref{equation:top-smooth-isotopy-concordance-R_0}, by partitioning $\cR^\bullet_M(G)$ and $\cC^\bullet_M(G)$ into the disjoint union of $\cR^\bullet_M(G,R_\alpha)$ and $\cC^\bullet_M(G,R_\alpha)$, where the spheres $R_\alpha$ are representatives of homotopy classes of smooth spheres having $G$ as a geometric dual.
  (For $\bullet=\top$, use the Light Bulb Smoothing Theorem~\ref{theorem:light-bulb-smoothing} to obtain the partition.)
\end{proof}

Now, let $M$ be a topological 4-manifold.
By Theorem~\ref{theorem:fq-top}, we have the following commutative diagram.
\begin{equation}
  \begin{tikzcd}[row sep=tiny]
    \cR^\top_{M\setminus p_0}(G,R_0) \ar[r] \ar[dd]
    & \cC^\top_{M\setminus p_0}(G,R_0)
      \ar[dd] \ar[rd,"{\FQ^\top_{M\setminus p_0}(R_0,-)}",end anchor={[yshift=-.8ex]north west},pos=.3]
    \\
    & & \bF_2 T_M / \mu(\pi_3(M))
    \\
    \cR^\top_M(G,R_0) \ar[r] 
    & \cC^\top_M(G,R_0) \ar[ru,"{\FQ^\top_M(R_0,-)}"',end anchor={[yshift=.8ex]south west},pos=.3]
  \end{tikzcd}
  \label{equation:top-isotopy-concordance-R_0}
\end{equation}

In~\eqref{equation:top-isotopy-concordance-R_0}, note that $\mu(\pi_3(M))$ is equal to $\mu(\pi_3(M\setminus p_0))$ as subgroups in $\bF_2 T_M = \bF_2 T_{M\setminus p_0}$, similarly to Lemma~\ref{lemma:dax-homomorphism-on-embedded-3-sphere}, since $\mu$ vanishes on $\pi_3$ elements represented by embedded spheres (see the definition of ~$\mu$ in~\ref{subsection:generic-maps-fq-top} in the appendix).

Also, consider the natural map
\begin{equation}
  \cR^\top_M(G) \to \cC^\top_M(G).
  \label{equation:top-isotopy-concordance-all}
\end{equation}

\begin{corollary}
  \label{corollary:sphere-in-top-4-manifold-isotopy-concordance}
  All maps in~\eqref{equation:top-isotopy-concordance-R_0} and~\eqref{equation:top-isotopy-concordance-all} are bijections.
\end{corollary}

In particular, for spheres with a common geometric dual in a topological 4-manifold, concordance implies isotopy.
This is a topological category version of~\cite[Corollary~1.4]{Schneiderman-Teichner:2022-1}.

\begin{proof}
  The proof for~\eqref{equation:top-isotopy-concordance-R_0} is identical to that of Corollary~\ref{corollary:disk-top-isotopy-classification}:
  since $M\setminus p_0$ has a smooth structure, the top horizontal arrow and $\FQ^\top_{M\setminus p_0}(R_0,-)$ are bijective by Corollary~\ref{corollary:top-equal-smooth-sphere-isotopy-concordance}.
  The vertical maps are surjective since we can isotope a sphere to avoid~$p_0$.
  By commutativity, it follows that all maps in~\eqref{equation:top-isotopy-concordance-R_0} are bijective.
  For~\eqref{equation:top-isotopy-concordance-all}, use the partition argument that we previously used for Corollary~\ref{corollary:top-equal-smooth-disk-isotopy} and~\eqref{equation:top-smooth-isotopy-concordance-all}.
\end{proof}

\section{More applications in the topological category}
\label{section:more-applications}

In this section, we discuss more applications in the topological category, which generalize previously known smooth category results in the literature.

The following theorem states a topological category version of Gabai's smooth 4D light bulb theorem~\cite{Gabai:2020-1}.
It is readily obtained as a consequence of Corollary~\ref{corollary:sphere-in-top-4-manifold-isotopy-concordance}.

\begin{theorem}[Gabai's light bulb theorem in the topological category]
  \label{theorem:top-light-bulb-theorem}
  If $M$ is a topological 4-manifold such that $\pi_{1}(M)$ has no 2-torsion, then two spheres in $M$ with a common geometric dual are homotopic if and only if they are isotopic.
\end{theorem}

Let $\Diff_0(M)$ and $\Homeo_0(M)$ be the groups of diffeomorphisms and homeomorphisms on $M$ which are properly homotopic to the identity, respectively.
Here, a proper homotopy means a homotopy through self maps of the pair $(M,\partial M)$.
Gabai shows that $\pi_0(\Diff_0(S^2 \times D^2)) \cong \pi_0(\Diff_0(B^4))$ using his smooth light bulb theorem~\cite{Gabai:2020-1}.
In the topological category, using Theorem~\ref{theorem:top-light-bulb-theorem} instead of the smooth light bulb theorem, his argument is carried out to show $\pi_0(\Homeo_0(S^2 \times D^2)) \cong \pi_0(\Homeo_0(B^4))$.
Since $\pi_0(\Homeo_0(B^4))$ is trivial, we obtain the following:

\begin{corollary}
  \label{corollary:top-mapping-class-S^2xD^2}
  $\pi_0(\Homeo_0(S^2\times D^2)) = 0$.
\end{corollary}

\begin{remark}
  \label{remark:alternative-approach}
  \begin{enumerate}
    \item \label{item:alternative-approach-without-surface-smoothing}
    An alternative proof of Theorem~\ref{theorem:top-light-bulb-theorem} can be given without using our surface smoothing results.
    After isotoping $R_0$ and $R_1$, there exists an almost smooth structure of $M$ with respect to which all of $G$, $R_0$ and $R_1$ are smooth (see Lemma~\ref{lemma:almost-smoothing-rel-surfaces} below), so that one can apply Gabai's smooth light bulb theorem with respect to that smooth structure.
    \item In~\cite{Orson-Powell:2022-1}, Orson and Powell studied the topological mapping class group rel boundary (i.e.,\ homeomorphisms fixing $\partial$ modulo isotopy rel~$\partial$) for simply connected 4-manifolds.
    One can obtain an alternative proof of Corollary~\ref{corollary:top-mapping-class-S^2xD^2} by combining their work with the result of Gluck~\cite{Gluck:1962-1} that homotopic homeomorphisms of $S^2\times S^1$ are isotopic.
  \end{enumerate}
\end{remark}

\begin{lemma}
  \label{lemma:almost-smoothing-rel-surfaces}
  Let $R_0$, $R_1$, $\ldots\,$,~$R_n$ be surfaces in a topological 4-manifold~$M$.
  Then $R_1$, $\ldots\,$,~$R_n$ can be isotoped so that there is an almost smooth structure of $M$ with respect to which $R_0$, $\ldots\,$,~$R_n$ are smooth.
\end{lemma}
\begin{proof}
  Apply topological transversality~\cite[9.5]{Freedman-Quinn:1990-1} inductively to isotope $R_1$, $\ldots\,$,~$R_n$ so that $R_0$, $\ldots\,$,~$R_n$ are pairwise transverse without triple intersections.
  As in the proof of the claim in Section~\ref{section:proof-light-bulb-smoothing}, $R_0\cup\cdots\cup R_n$ has a plumbed tubular neighborhood~$N$.
  Choose an almost smooth structure of the exterior $\overline{M\setminus N}$ by using \cite[8.2]{Freedman-Quinn:1990-1} and extend it to $M$ using the standard smooth structure of~$N$.
\end{proof}

The approach in Remark~\ref{remark:alternative-approach}~\ref{item:alternative-approach-without-surface-smoothing} which uses Lemma~\ref{lemma:almost-smoothing-rel-surfaces} also proves topological versions of earlier smooth category results of Auckly, H. J. Kim, Melvin, Ruberman and Schwartz \cite{Auckly-Kim-Melvin-Ruberman-Schwartz:2019-1}
and Klug and Miller~\cite{Klug-Miller:2021-1}.
We state these below.

\begin{theorem}[Topological version of \cite{Auckly-Kim-Melvin-Ruberman-Schwartz:2019-1}]
  Let $M$ be a simply connected topological 4-manifold, and 
  $R_0$ and $R_1$ be surfaces of the same genus in $M$ which have simply connected complements and represent the same homology class $\alpha\in H_2(M)$.
  If $\alpha$ is not the dual of the Stiefel-Whitney class $w_2$, $R_0$ and $R_1$ are isotopic in $M\# (S^2\times S^2)$.
  Otherwise, $R_0$ and $R_1$ are isotopic in $M\# (S^2 \widetilde\times S^2)$.
\end{theorem}

 \begin{theorem}[Topological version of \cite{Klug-Miller:2021-1}]
  Let $M$ be a topological 4-manifold and $R_{0}$ and $R_{1}$ be two homotopic spheres.
  If there is an immersed framed sphere which is geometrically dual to~$R_{0}$, then $R_{0}$ and $R_{1}$ are concordant if and only if $\FQ^\top_M(R_0,R_1)$ vanishes.
  \end{theorem}

\begin{remark} 
  It seems hard, if possible at all, to develop the Dax invariant for disks in topological 4-manifolds by an alternative approach similar to Remark~\ref{remark:alternative-approach}~\ref{item:alternative-approach-without-surface-smoothing}, i.e.,\ by choosing (almost) smooth structures cleverly without using our disk smoothing results.
  Even if such an alternative approach worked, the choice of an (almost) smooth structure would depend on the given topological disks.
  On the contrary, our development in Section~\ref{subsection:dax-for-top-disks} and~\ref{subsection:dax-in-top-4-manifold} uses a fixed (almost) smooth structure.
  This is essential in proving ``topological = smooth'' results in smooth 4-manifolds, e.g.,\ Corollaries~\ref{corollary:top-equal-smooth-disk-isotopy}, \ref{corollary:top-equal-smooth-sphere-isotopy} and~\ref{corollary:top-equal-smooth-concordance-isotopy}, since one has to stick to \emph{the} smooth structure of a given smooth 4-manifold.
\end{remark}

We finish this section with another example that illustrates the usefulness of the topological Dax invariant.
In the smooth category, Schwartz partially generalizes earlier results of Gabai~\cite{Gabai:2021-1} and Kosanovi\'{c}-Teichner~\cite{Kosanovic-Teichner:2024-1} by allowing a geometric dual not contained in the boundary of the ambient 4-manifold~\cite{Schwartz:2021-1}.
We show a topological version of her result.
To state it, let $M$ be a topological 4-manifold and $G$ be a framed sphere in~$M$.
Fix an almost smooth structure of $M$ with respect to which $G$ is smooth, to define $\Dax^\top_M(-,-)$ by applying Definition~\ref{definition:dax-in-top-4-manifold}.

\begin{corollary}
  \label{corollary:schwartz-top}
  Let $D_0$ and $D_1$ be topological disks with $G$ as a geometric dual, which admit a homotopy $D_0\simeq D_1$ rel $\partial$ supported away from~$G$.
  Suppose that the inclusion induces an isomorphism $\pi_1(M\setminus G) \to \pi_1(M)$ and the Dax homomorphism $d\colon \pi_3(M)\to \Z[\pi_1(M)\setminus 1]^\sigma$ is trivial so that $\Dax^\top_M(D_0,D_1)$ lies in $\Z[\pi_1(M)\setminus 1]^\sigma$.
  Then, $D_0$ and $D_1$ are topologically isotopic if and only if $\Dax^\top_M(D_0,D_1)$ vanishes.
\end{corollary}

\begin{proof}
  The only if direction is true by Theorem~\ref{theorem:dax-in-top-4-manifold-isotopy}.
  (This does not require the existence of a common geometric dual.)
  For the converse, suppose $\Dax^\top_M(D_0,D_1) = 0$. 
  Isotope $D_i$ to avoid the singular point $p_0$ of the almost smooth structure, and apply the Light Bulb Smoothing Theorem~\ref{theorem:light-bulb-smoothing} to obtain smoothings $D'_i$ of $D_i$ in $M\setminus p_0$.
  By Definition~\ref{definition:dax-in-top-4-manifold}, $\Dax^\sm_{M\setminus p_0}(D_0',D_1') = \Dax^\top_M(D_0,D_1) = 0$.
  Also, $d\colon \pi_1(M\setminus p_0) \to \Z[\pi_1(M) \setminus 1]^\sigma$ vanishes by Lemma~\ref{lemma:dax-homomorphism-on-embedded-3-sphere}.
  It follows that $D_0'$ and $D_1'$ are smoothly isotopic in $M\setminus p_0$, by the result of Schwartz~\cite{Schwartz:2021-1}.
  Therefore, $D_0$ and $D_1$ are isotopic in~$M$.
\end{proof}

\appendix
\def\sectionname{}\let\temp=\thesection\def\thesection{Appendix}
\section{High-dimensional smoothing, immersion and topological Freedman-Quinn invariant}
\label{section:smoothing-immersion-top-fq}
\let\thesection=\temp

In the smooth category, Schneiderman and Teichner provided a detailed treatment of the Freedman-Quinn invariant~\cite{Freedman-Quinn:1990-1,Stong:1993-1}, including its definition and basic properties in~\cite[Section~4]{Schneiderman-Teichner:2022-1}.
See also~\cite{Klug-Miller:2021-1,Klug-Miller:2022-1}.
The definition of the invariant is based on counting self-intersections of maps of 
$S^2\times I$ and $S^3$ into certain 6-manifolds, after perturbing involved maps to self-transverse maps.
The proof of the well-definedness depends on a similar perturbation to generic maps.
These arguments rely on standard results of transversality and singularity theory for smooth maps.

The main goal of this appendix is to provide a concrete treatment in the topological category.
It appears to be known to experts that the Freedman-Quinn invariant can also be defined in the topological category, but no details were given in the literature.
To fill in this omission, we describe technical details, especially how one can perturb involved maps to generic maps in the topological category, using topological immersion theory of Lees~\cite{Lees:1969-1} and high dimensional smoothing theory of Kirby and Siebenmann~\cite{Kirby-Siebenmann:1977-1}.

In this appendix, manifolds are topological unless stated otherwise.

\subsection{Generic maps and Freedman-Quinn invariant}
\label{subsection:generic-maps-fq-top}

We begin with the notion of a generic immersion.
Let $P$ be a 6-manifold and $Y$ be either $S^3$ or $S^2\times I$.
We say that a map $f\colon Y\to P$ is a \emph{generic immersion} if $f$ embeds $\partial Y$ into $\partial P$ and $f$ is a local embedding which fails to be one-to-one only at finitely many transverse double points.
When $P$ is smooth, a \emph{smooth generic immersion} $Y\to P$ is a local smooth embedding which is a topological generic immersion.
(The 3-manifold $Y$ is equipped with its unique smooth structure.)
It is known that if a smooth map $f\colon Y\to P$ is a smooth embedding on $\partial Y$, then $f$ is homotopic rel $\partial Y$ to a smooth generic immersion (e.g. see~\cite[Theorem~2.5]{Haefliger:1961-1}).
The following is a topological analog for our case.

\begin{lemma}
  \label{lemma:homotopy-to-generic-immersion}
  A map $f\colon Y \to P$ that embeds $\partial Y$ into~$\partial P$ is homotopic rel $\partial Y$ to a generic immersion $g\colon Y\to P$.
  In fact, there is a smooth structure on a neighborhood $W$ of $g(Y)$ in $P$ such that $g\colon Y\to W$ is a smooth generic immersion.
  If $f$ is based, the homotopy can be assumed to be based.
\end{lemma}

Raising the dimensions by one, consider a map $F\colon X \to Q$, where $Q$ is a 7-manifold and $X=S^3\times I$ or $S^2\times I^2$ (i.e.\ $X=Y\times I$ where $Y$ is as above).
Similarly to the above, if $Q$ is smooth and $F\colon X\to Q$ is a smooth map that restricts to a smooth generic immersion $\partial X \to \partial Q$, then $F$ is homotopic rel $\partial X$ to a smooth generic map (e.g.\ see~\cite[Theorem~2.5]{Haefliger:1961-1}).
Here, $F\colon X\to P$ is a \emph{smooth generic map} if $F$ is a smooth immersion except at a finite set $S\subset \inte(X)$ of cusp singularities, $F$ has no triple points, double points of $F$ are not in $F(S)$, and $F$ is self-transverse at each double point.
Lemma~\ref{lemma:homotopy-to-generic-map} below is a topological analog.
In the topological category, we say that $F\colon X\to Q$ is a \emph{generic map} if $F$ satisfies the local properties of a smooth generic map.

\begin{lemma}
  \label{lemma:homotopy-to-generic-map}
  Suppose that $F\colon X \to Q$ is a map such that $F|_{\partial X}\colon \partial X \to \partial Q$ is a generic immersion.
  Then $F$ is homotopic rel $\partial X$ to a generic map $G\colon X\to Q$.
  In fact, there is a smooth structure on a neighborhood $W$ of $G(X)$ in $Q$ such that $G\colon X\to W$ is a smooth generic immersion.
\end{lemma}

The proofs of Lemmas~\ref{lemma:homotopy-to-generic-immersion} and~\ref{lemma:homotopy-to-generic-map} are given in Section~\ref{subsection:immersion-smoothing-generic-maps} of this appendix.

Once we have Lemmas~\ref{lemma:homotopy-to-generic-immersion} and~\ref{lemma:homotopy-to-generic-map}, we can carry out the smooth category development of the Freedman-Quinn invariant~\cite[Section~4]{Schneiderman-Teichner:2022-1}, without essential changes, in the topological category.
For the reader's convenience, we outline it below.

Let $Y=S^3$ or $S^2\times I$ and $P$ be a topological 6-manifold, as above.
Fix basepoints, and let $f\colon Y\to P$ be a based map that embeds $\partial Y$ into~$\partial P$.
By Lemma~\ref{lemma:homotopy-to-generic-immersion}, $f$ is based homotopic rel $\partial Y$ to a based generic immersion $f'\colon Y\to P$.
For a double point $p$ of $f'$, the sign $\epsilon(p)\in \{1,-1\}$ and the $\pi_1$ element $g(p)\in \pi=\pi_1(P)$ is defined as usual.
Switching the involved sheets changes $\epsilon(p)g(p)$ to $-\epsilon(p)g(p)^{-1}$, so the self-intersection invariant is defined as follows:
\[
  \mu(f) = \sum_p \epsilon(p)g(p) \in \Z[\pi]/\langle 1,g+g^{-1}\rangle.
\]

In the smooth category, it is known that $\mu(f)$ is invariant under based homotopy of $f$ rel~$\partial Y$.
More generally, $\mu(f_0)=\mu(f_1)$ if $f_0$, $f_1\colon Y\to P$ are smooth generic immersions into a smooth 6-manifold $P$ and if there is a map $F\colon Y\times I\to P\times I$ such that $F|_{Y\times i} = f_i\times i$ ($i=0,1$) and $F|_{\partial Y\times I} = f_0|_{\partial Y}\times\id$.
A sketch of a proof is as follows.
We may assume that $F$ is a smooth generic map.
Then double points of $F$ form a 1-manifold bounded by double points of $F|_{\partial(Y\times I)}$ and cusp singular points of~$F$, so these points are paired up by arc components.
If an arc component $\beta$ joins two double points $p_1$ and $p_2$, $\epsilon(p_1)g(p_1) = \epsilon(p_2)g(p_2)$.
If $\beta$ joins a double point $p$ and a cusp, $g(p)$ vanishes since a double point loop near a cusp is null-homotopic.
It follows that $\mu(f_1)-\mu(f_0)=\mu(F|_{\partial(Y\times I)})=0$ in $\Z[\pi]/\langle 1,g+g^{-1}\rangle$.

In the topological category, Lemma~\ref{lemma:homotopy-to-generic-map} ensures that $\mu(f)$ is invariant under based homotopy of~$f$.
For a topological 6-manifold $P$, a map $F\colon Y\times I \to P\times I$ as above is homotopic to a (topological) generic map by Lemma~\ref{lemma:homotopy-to-generic-map}, so that the above argument can be applied.

Similarly, for two maps $f_0$, $f_1\colon Y\to P$ which embeds $\partial Y$ into~$\partial P$, the intersection invariant $\lambda(f_0,f_1)\in \Z[\pi]$ can defined in the topological category.
Homotope $f_i$ to a generic immersion rel $\partial$ using Lemma~\ref{lemma:homotopy-to-generic-immersion}, isotope $f_i$ to assume that double points of $f_i$ are not in $f_{1-i}(Y)$, and apply topological transversality~\cite[9.5]{Freedman-Quinn:1990-1} to the submanifolds $f_i(Y)\setminus\{$double points$\}$ to make $f_0$ and $f_1$ transverse.
Count transverse intersections $p$ of $f_0(Y)$ and $f_1(Y)$ with signs $\epsilon(p)$ and associated $\pi_1$ elements $g(p)$, to define $\lambda(f_0,f_1)=\sum_p \epsilon(p)g(p) \in \Z[\pi]$.
The above equating argument that uses Lemma~\ref{lemma:homotopy-to-generic-map} shows that $\lambda(f_0,f_1)$ is invariant under homotopy of $f_i$ rel~$\partial Y$ in the topological category.

When $P=M\times \R\times I$ where $M$ is a 4-manifold, the value of $\mu(f)$ lies in the subgroup $\bF_2 T_M \subset \Z[\pi_1(M)] / \langle 1,g+g^{-1}\rangle$, where $\bF_2 T_M$ is the vector space over $\bF_2$ generated by the set $T_M$ of 2-torsion elements in $\pi_1(M)$, by~\cite[Lemma~4.1]{Schneiderman-Teichner:2022-1}.
Especially, when $Y=S^3$, a self-intersection invariant $\mu\colon \pi_3(M)\to \bF_2 T_M$ is defined by $\mu([f])=\mu(f)$ where $f\colon S^3\to M\times\R\times I$ is a based generic immersion representative of a given element $[f]\in \pi_3(M)$.
By~\cite[Lemma~4.1]{Schneiderman-Teichner:2022-1}, $\mu$ is a homomorphism.

Let $R_0$ and $R_1$ be homotopic spheres embedded in a 4-manifold~$M$, which have a common geometric dual~$G$.
Suppose $R_0$ and $R_1$ agree near~$G$.
Let $H\colon S^2\times I \to M\times\R\times I$ be a map which restricts to an embedding $S^2\times i \hookrightarrow M\times 0\times i$ onto $R_i=R_i\times 0\times i$ for $i=0,1$.
(For instance, the track of a homotopy $R_0\simeq R_1$ is such a map.)
By the above, $\mu(H)$ is invariant under homotopy of~$H$ rel $\partial$.
Now, the \emph{Freedman-Quinn invariant} for $(R_0,R_1)$ is defined by $\FQ(R_0,R_1)=\mu(H)\in \bF_2 T_M / \mu(\pi_3(M))$.
It is well-defined, independent of the choice of~$H$, by the arguments in~\cite[Lemma~4.5]{Schneiderman-Teichner:2022-1}, \cite[Proposition~4.9]{Klug-Miller:2021-1} and \cite[Proposition~3.3]{Klug-Miller:2022-1}.

We remark that the topological version of the intersection invariant $\lambda$ defined above is needed to carry out the arguments of~\cite[Lemmas~4.1 and 4.5]{Schneiderman-Teichner:2022-1} in the topological category.

When $R_0$ and $R_1$ are concordant, we can compute $\FQ(R_0,R_1)$ using a concordance between $R_0$ and $R_1$, which is an embedding $H\colon S^2\times I \to M\times\R\times I$.
Obviously, an embedding has vanishing~$\mu$.
Therefore $\FQ(R_0,R_1) = \mu(H) = 0$ if $R_0$ and $R_1$ are concordant.
It is readily seen from the definition that the additivity $\FQ(R_0,R_2)=\FQ(R_0,R_1)+\FQ(R_1,R_2)$ holds.
Thus, it follows that $\FQ(R_0,R_1)$ is determined by the concordance classes of $R_0$ and $R_1$.

This shows that Theorem~\ref{theorem:fq-top} in the body of this paper holds.

\subsection{Immersions, smoothing and proofs of Lemmas~\ref{lemma:homotopy-to-generic-immersion} and~\ref{lemma:homotopy-to-generic-map}}
\label{subsection:immersion-smoothing-generic-maps}

The proofs of Lemmas~\ref{lemma:homotopy-to-generic-immersion} and~\ref{lemma:homotopy-to-generic-map} given below use Lee's immersion theorem in~\cite{Lees:1969-1} (see also \cite{Lashof:1969-1,Lashof:1971-1}).
For our purpose, it suffices to use the following consequence of the immersion theorem.
For a map $f\colon V\to P$ between manifolds, denote by $df \colon \tau_V \to \tau_P$ the induced bundle map on the tangent (micro)bundles.

\begin{theorem}[{A consequence of \cite[Immersion Theorem, p.~189]{Lashof:1971-1}}]
  \label{theorem:top-immersion}
  Let $V$ and $P$ be $n$-manifolds without boundary, $C\subset V$ be closed, and $h\colon U\to P$ be an immersion of a neighborhood $U$ of~$C$.
  If there is a bundle map $\phi\colon \tau_V \to \tau_P$ such that $\phi=dh$ near~$C$, then there exists an immersion $g_1\colon V\to P$ such that $g_1=h$ near $C$ and $dg_1$ and $\phi$ are homotopic as bundle maps.
  (Consequently, $g_1$ is homotopic to the map $V\to P$ induced by~$\phi$.)
\end{theorem}

We will also use the following high-dimensional smoothing theorem of Kirby and Siebenmann~\cite{Kirby-Siebenmann:1977-1}.
For a map $g\colon V\to W$ from a smooth manifold $V$ to a (topological) manifold~$W$, we say that $g$ is \emph{weakly smoothly self-transverse on $C$} if there is a neighborhood $V_C$ of $C$ in $V$ and a smooth structure on $g(V_C)$ such that $g|_{V_c}\colon V_C\to g(V_C)$ is a smooth immersion and $g|_C$ is smoothly self-transverse.

\begin{theorem}[{\cite[Essay III, C.6]{Kirby-Siebenmann:1977-1}}]
  \label{theorem:ks-smoothing}
  Let $V$ and $W$ be $n$-manifolds, $V$ smooth, $g\colon V\to W$ an immersion, $M\subset V$ a compact smooth $m$-submanifold such that $g|_M$ is proper, and $C\subset M$ a closed subset on which $g$ is weakly smoothly self-transverse.
  Suppose $n\ge 5$ and $2(m-1) \le n$.
  Then there is an arbitrarily small regular homotopy of $g$ to $g_1\colon V\to C$ rel $C$ such that $g_1$ is weakly smoothly self-transverse on~$M$.
\end{theorem}

\begin{proof}[Proof of Lemma~\ref{lemma:homotopy-to-generic-immersion}]
  Identifying collars of $\partial Y$ and $\partial P$ with $\partial Y\times I$ and $\partial P\times I$, we may assume that the given $f\colon Y\to P$ restricts to $f|_{\partial Y}\times\id$ on the collar.

  \begin{claim}
    \hskip-1pt There exist a smooth 3-dimensional vector bundle $\xi$ over $Y$ \kern-1.6pt and a bundle map $\phi\colon \tau_{E(\xi)} \to \tau_P$ such that the induced map $E(\xi)\to P$ extends $f\colon Y\to P$ and restricts to a topological normal bundle $E(\xi|_{\partial Y}) \hookrightarrow \partial P$ of $f(\partial Y)$ in~$\partial P$.
    (Here $Y\subset E(\xi)$ via the zero section.)
  \end{claim}

  In what follows, denote by $B\TOP_n$ and $BO_n$ the classifying spaces for topological $\R^n$-bundles and $n$-dimensional vector bundles, and denote by $\TOP_n/O_n$ the (homotopy) fiber of the natural map $BO_n \to B\TOP_n$.

  \begin{step-named}[Case~1]
    Suppose $Y=S^3$.
    We have $\pi_3(BO_6)=\pi_2(O_6)=\pi_2(O_3)=0$.
    Also, by~\cite[Essay V, p~246]{Kirby-Siebenmann:1977-1}, $\pi_2(\TOP_6/O_6) = 0$.
    By the long exact sequence for the fibration $\TOP_6/O_6 \hookrightarrow BO_6 \to B\TOP_6$, it follows that $\pi_3(B\TOP_6)=0$.
    Thus $f^*(\tau_P)=\varepsilon^6_Y$, a trivial bundle over~$Y$.
    Let $\xi=\varepsilon^3_Y$.
    Let $g=f\circ p\colon E(\xi)=Y\times \R^3 \to P$, where $p\colon Y\times \R^3\to Y$ is the projection.
    Then $g^*(\tau_P) =\epsilon^6_{Y\times\R^3} = \tau_{Y\times \R^3}$, so $g$ is covered by a bundle map $\tau_{Y\times \R^3} = g^*(\tau_P) \to \tau_P$ as claimed.
  \end{step-named}

  \begin{step-named}[Case~2]
    Suppose $Y=S^2\times I$.
    Let $\nu$ over $\partial Y$ be the pullback of the normal 3-dimensional vector bundle of $f(\partial Y)\subset \partial P$.
    (A $k$-submanifold in an $m$-manifold with $k\le 3$ and $m\ge 5$ always has a normal vector bundle, since it has a stable normal $\TOP$ bundle and $BO_3 \to BO \to B\TOP$ is 3-connected by~\cite[Essay V, p~246]{Kirby-Siebenmann:1977-1}.)

    It suffices to show that $\nu$ extends to a 3-dimensional vector bundle $\xi$ over~$Y$.
    In fact, if $\nu$ extends to~$\xi$, $E(\nu) \cup (Y\times I)$ is a strong deformation retract of~$E(\xi)$, so $j \cup f\colon E(\nu) \cup (Y\times I) \to P$ extends to a map $g\colon E(\xi) \to P$, where $j$ is the normal bundle inclusion $E(\nu) \hookrightarrow \partial P$.
    Then $g^*(\tau_P)|_{S^2\times 0} = (f|_{S^2\times 0})^*(\tau_{\partial P}\oplus \varepsilon^1) = \tau_{E(\nu)}|_{S^2\times 0} \oplus \varepsilon^1 = \tau_{E(\xi)}|_{S^2\times 0}$.
    It follows that $g^*(\tau_P) = \tau_{E(\xi)}$ since $S^2\times 0 \hookrightarrow Y=S^2\times I$ is a homotopy equivalence.

    For $i=0,1$, $f^*(\tau_P)|_{S^2\times i} = (f|_{S^2\times i})^* (\tau_{\partial P} \oplus \varepsilon^1) = \tau_{S^2\times i} \oplus \nu|_{S^2\times i} \oplus \varepsilon^1 = \nu|_{S^2\times i}\oplus \varepsilon^3$.
    Since $\pi_2(BO_6) = \pi_2(BO_3)$, it follows that $\nu|_{S^2\times 0} \cong \nu|_{S^2\times 1}$.
    Therefore, $\nu$ extends to $\xi:=\nu|_{S^2\times 0}\times I$ over $Y=S^2\times I$.
    This completes the proof of the claim.
  \end{step-named}

  Now, write $V=E(\xi)$.
  We may assume that $g\colon V\to P$ is an embedding on a collar $\partial V\times I$ of $\partial V\times 0 = \partial V = E(\xi|_{\partial Y})$ in~$V$.
  Let $V_0=V\setminus \partial V$, $Q_0=Q\setminus \partial Q$, $C=\partial V\times (0,\frac12] \subset V_0$, and apply Theorem~\ref{theorem:top-immersion} to $(V_0, P_0, C, \phi\colon \tau_{V_0} \to \tau_{P_0})$, to homotope $g$ rel $\partial V\times[0,\frac12]$ to an immersion $g\colon V\to P$.
  Apply Theorem~\ref{theorem:ks-smoothing}, to regularly homotope $g$ rel a neighborhood of $\partial V$ so that there exists a neighborhood $V_Y$ of $Y$ in $V$ and a smooth structure on $g(V_Y)$ for which $g|_{V_Y}\colon V_Y \to g(V_Y)$ is a smooth immersion.
  Finally, by a homotopy rel a neighborhood of $\partial V$, change the smooth map $g|_Y\colon Y\to g(V_Y)$ to a smooth generic immersion.
\end{proof}

\begin{proof}[Proof of Lemma~\ref{lemma:homotopy-to-generic-map}]
  Let $\nu$ over $\partial X$ be the pullback of the normal 3-dimensional vector bundle of $F(\partial X)$ in~$\partial Q$.
  As in Case~2 of the proof of Lemma~\ref{lemma:homotopy-to-generic-immersion}, it suffices to show that $\nu$ extends to a vector bundle $\xi$ over~$X$.
  The remaining part of the proof of Lemma~\ref{lemma:homotopy-to-generic-immersion} is carried out in the same way.

  Since all involved manifolds are oriented, the normal bundle $\nu$ is oriented, and thus $\nu$ is classified by a map $\partial X \to BSO_3$.
  The obstruction to extending $\partial X \to BSO_3$ to $X$ lies in $H^k(X,\partial X;\pi_{k-1}(BSO_3))$.
  We have $\pi_0(BSO_3)=0$ and $\pi_{k-1}(BSO_3)=\pi_{k-2}(SO_3)=0$ for $k=2, 4$.
  The relative complex $(X,\partial X)$ has one 1-cell and one 4-cell if $X=S^3\times I$, and one 2-cell and one 4-cell if $X=S^2\times I^2$.
  Thus all obstructions vanish and $\nu$ extends to a vector bundle $\xi$ over~$X$.
\end{proof}

\bibliographystyle{amsalpha}
\renewcommand{\MR}[1]{}

\hbadness 10000
\bibliography{research}

\end{document}